\documentclass[12pt]{amsart}
\usepackage{amsmath,amsthm,amssymb,verbatim,url,color, enumitem,hyperref}
\usepackage[margin=1.25in]{geometry}
{\providecommand{\noopsort}[1]{} 

\theoremstyle{plain}
\newtheorem{theorem}{Theorem}[section]
\newtheorem{corollary}[theorem]{Corollary}\newtheorem*{nonumbercorollary}{Corollary}
\newtheorem{lemma}[theorem]{Lemma}
\newtheorem{proposition}[theorem]{Proposition}

\newtheorem{question}[theorem]{Question}
\newtheorem*{motivatingquestion}{Motivating Question}

\newtheorem{main}{Theorem}

\newtheorem*{pro:LowDimIntro}{Proposition \ref{pro:LowDimIntro}}
\newtheorem*{pro:PWSChomogeneous}{Proposition \ref{pro:PWSChomogeneous}}
\newtheorem*{thm:GroveSearle}{Theorem \ref{thm:GroveSearle}}
\newtheorem*{thm:WilkingHomotopy}{Theorem \ref{thm:WilkingHomotopy}}
\newtheorem*{thm:Connectedness}{Theorem \ref{thm:Connectedness}}

\theoremstyle{definition}
\newtheorem{definition}[theorem]{Definition}\newtheorem*{nonumberdefinition}{Definition}

\theoremstyle{remark}
\newtheorem{remark}[theorem]{Remark}
\newtheorem{example}[theorem]{Example}
\numberwithin{equation}{section}

\newcommand{\ep}{\varepsilon}\newcommand{\la}{\lambda}

\newcommand{\Z}{\mathbb{Z}}\newcommand{\Q}{\mathbb{Q}}
\newcommand{\R}{\mathbb{R}}\newcommand{\C}{\mathbb{C}}\newcommand{\HH}{\mathbb{H}}

\newcommand{\pp}{\mathbb{P}}
\newcommand{\s}{\mathbb{S}}

\DeclareMathOperator{\SO}{SO}\DeclareMathOperator{\Or}{O}
\DeclareMathOperator{\SU}{SU}

\DeclareMathOperator{\cod}{cod}

\DeclareMathOperator{\Hess}{Hess}

\DeclareMathOperator{\vol}{vol}
\DeclareMathOperator{\area}{area}
\DeclareMathOperator{\Ric}{Ric}
\DeclareMathOperator{\scal}{scal}
\DeclareMathOperator{\sym}{sym}

\newcommand{\inner}[1]{\left\langle #1 \right\rangle}

\newcommand{\of}[1]{\left(#1\right)}
\newcommand{\ofsq}[1]{\left[#1\right]}

\newcommand{\floor}[1]{\left\lfloor #1 \right\rfloor}

\newcommand{\st}{~|~}

\newcommand{\embedded}{\hookrightarrow}
\newcommand{\id}{\mathrm{id}}

\title{Positive Weighted Sectional Curvature}
\author{Lee Kennard}
\address{Department of Mathematics, University of California, Santa Barbara, CA 93106}
\email{kennard@math.ucsb.edu}
\urladdr{www.math.ucsb.edu/~kennard}
\author{William Wylie}
\address{215 Carnegie Building\\
Dept. of Math, Syracuse University\\
Syracuse, NY, 13244.}
\email{wwylie@syr.edu}
\urladdr{https://wwylie.expressions.syr.edu}
\date{\today}

\begin{document}

\begin{abstract}
In this paper, we give a new generalization of positive sectional curvature called \emph{positive weighted sectional curvature}. It depends on a choice of Riemannian metric and a smooth vector field. We give several simple examples of Riemannian metrics which do not have positive sectional curvature but support a vector field that gives them positive weighted curvature.   On the other hand, we generalize a number of the foundational results for compact manifolds with positive sectional curvature to positive weighted curvature. In particular, we prove generalizations of Weinstein's theorem, O'Neill's formula for submersions, Frankel's theorem, and Wilking's connectedness lemma. As applications of these results, we recover weighted versions of topological classification results of Grove--Searle and Wilking for manifolds of high symmetry rank and positive curvature. 
\end{abstract}


\maketitle


Understanding Riemannian manifolds with positive sectional curvature is a deep and notoriously difficult problem in Riemannian geometry. A common approach in mathematics to such problems is to generalize it to a more flexible one and study this generalization with the hope that it will shed light on the harder original problem. Indeed, there are a number of generalizations of positive sectional curvature that have been studied. The most obvious is non-negative sectional curvature, but other conditions such as quasi-positive or almost positive curvature have been studied in the literature (see \cite{Ziller07,KerrTapp14} and references therein)

In this paper we propose a different approach to generalizing positive curvature that depends on choosing a positive, smooth density function, denoted by $e^{-f}$, or a smooth vector field $X$. Our motivation for considering such a generalization is the corresponding theory of Ricci curvature for manifolds with density, which was studied by Lichnerowicz \cite{Lich1, Lich2} and was later generalized and popularized by Bakry--Emery and their collaborators \cite{BakryEmery}. There are too many recent results in this area to reference all of them here, but some that are more relevant to this article are \cite{Lott03, Morgan05, Morgan09, MunteanuWang12, WeiWylie09}. Also see Chapter 18 of \cite{Morganbook} and the references therein.  

For a triple $(M^n,g,X)$, where $(M,g)$ is a Riemannian manifold and $X$ is a smooth vector field, the $m$--Bakry--Emery Ricci tensor is
	\[ \Ric_X^m = \Ric +\frac{1}{2} L_X g - \frac{X ^{\sharp}\otimes X^{\sharp}}{m},  \]
 where $m$ is a constant that is also allowed to be  infinite, in which case we write $\Ric_X^{\infty} = \Ric_X = \Ric +\frac{1}{2} L_X g$.  For a manifold with density, we set $X = \nabla f$ and write $\Ric_f^m =  \Ric +\mathrm{Hess} f  - \frac{ df \otimes df}{m}$.  
 
 The Bakry--Emery Ricci tensors  come up in many areas of geometry and analysis  including optimal transport \cite{LottVillani09, SturmI, SturmII, SturmVon}, the isoperimetric inequality \cite{Morgan05}, and the Ricci flow \cite{Perelman}.  Our definition of positive weighted sectional curvature, which looks similar to the Bakry--Emery Ricci tensors, is the following. 
 
\begin{nonumberdefinition}\label{def:pwsc}
A Riemannian manifold $(M,g)$ equipped with a vector field $X$ has \textit{positive weighted sectional curvature} if for every point $p\in M$, every $2$-plane $\sigma \subseteq T_pM$, and every unit vector $V \in \sigma$, 
	\begin{itemize}
	\item $\sec(\sigma) + \frac{1}{2} (L_X g)(V,V) > 0$, or
	\item $X = \nabla f$ and $\sec(\sigma) + \mathrm{Hess} f(V,V) + df(V)^2 > 0$ for some function $f$.
	\end{itemize}
\end{nonumberdefinition}

Note that a Riemannian manifold with positive sectional curvature admits positive weighted sectional curvature, where $X$ is chosen to be zero. This converse to this statement does not hold, as we show by example in Propositions \ref{pro:RotSym} and \ref{pro:CPnExample}. For additional examples that further illustrate the difference between these notions, we refer to Section \ref{sec:Examples}.

This definition is motivated by earlier work of the second author \cite{Wylie-pre} where generalizations of classical results such as the classification of constant curvature spaces, the theorems of Cartan--Hadamard, Synge, and Bonnet--Myers, and the (homeomorphic) quarter-pinched sphere theorem are proven for manifolds with density. 

There are a number of reasons why positive weighted sectional curvature is a natural generalization of positive sectional curvature. We will discuss this in more detail in Section \ref{sec:Preliminaries}. For example, we observe in Section \ref{sec:Preliminaries} that the following low-dimensional result holds (see Theorem \ref{thm:pi1finite} and the following remarks). It follows from earlier work of the second author \cite{Wylie08, Wylie-pre}.  

\begin{main}\label{thm:LowDimIntro} Suppose $M$ is a compact manifold of dimension two or three. If $M$ admits a metric and a vector field with positive weighted sectional curvature, then $M$ is diffeomorphic to a spherical space form. \end{main}

This raises the following motivating question in higher dimensions. 

\begin{motivatingquestion}
If $(M^n, g, X)$ is compact with positive weighted curvature, does $M$ admit a metric of positive sectional curvature? 
\end{motivatingquestion}

Theorem \ref{thm:LowDimIntro} shows the answer is ``yes" in dimension $2$ and $3$. On the other hand,  we show there are complete metrics with density on $\mathbb{R} \times T^n$ with positive weighted sectional curvature. By a theorem of Gromoll--Meyer \cite{GromollMeyer69}, $\mathbb{R} \times T^n$ does not admit a metric of positive curvature, so the answer is ``no" in the complete case. 

We approach this question by considering spaces with a high amount of symmetry.  Since the 1990s, when Grove popularized the approach, quite a lot of powerful machinery has been developed  for studying manifolds with positive curvature through symmetry. See the survey articles \cite{Wilking07, Grove09, Ziller14} for details as well as the many applications. 
 
  A first consideration is that a given vector field $X$ may not be invariant under the isometries of $g$. In Section \ref{sec:Averaging}, we deal with this issue by showing that, given a triple $(M,g,X)$ with positive weighted curvature and a compact group of isometries $G$ acting on $(M,g)$, it is always possible to change $X$ to $\widetilde{X}$ which  is invariant under $G$ so that $(M,g, \widetilde{X})$ has positive weighted sectional curvature.  The fact that  we can always assume that the density is invariant under a fixed compact subgroup of isometries will be a key observation in most of our results. In fact, it immediately gives the following result in the homogeneous case (see Proposition \ref{pro:CompactHomogeneous}). 
 
\begin{main}\label{thm:PWSChomogeneous} If a compact, homogeneous Riemannian manifold $(M,g)$ supports a gradient field $X = \nabla f$ such that $(M,g,X)$ has positive weighted curvature, then $(M,g)$ has positive sectional curvature. \end{main}

Simple examples show that this proposition is not true if the manifold is not compact (see Example \ref{Ex:Gaussian}).  In Section \ref{sec:Examples}, we also give examples of cohomogeneity one metrics on spheres and projective spaces that have positive weighted sectional curvature but not positive sectional curvature, so the homogeneous assumption cannot be weakened.

Another way to quantify that a Riemannian manifold has a large amount of symmetry is the symmetry rank, which is the largest dimension of a torus which acts effectively on $M$ by isometries.  Our main result regarding symmetry rank and positive weighted sectional curvature is an extension of the maximal symmetry rank theorem of Grove--Searle \cite{GroveSearle94} to positive weighted sectional curvature (see Theorem \ref{thm:GroveSearle}).

\begin{main}[Maximal symmetry rank theorem] \label{IntroGroveSearle}
Let $(M^n,g,X)$ be closed with positive weighted sectional curvature. If $T^r$ is a torus acting effectively by isometries on $M$, then $r \leq \floor{\frac{n+1}{2}}$. Moreover, if equality holds and $M$ is simply connected, then $M$ is homeomorphic to $\s^n$ or $\C\pp^{n/2}$.
\end{main}

In higher dimensions, Wilking has shown one can assume less symmetry and still obtain a homotopy classification \cite[Theorem 2]{Wilking03}. We also give an extension of this result (see Theorem \ref{thm:WilkingHomotopy}).

\begin{main}[Half-maximal symmetry rank theorem] \label{IntroWilkingHomotopy}
Let $(M^n,g,X)$ be closed and simply connected with positive weighted sectional curvature. If $M$ admits an effective, isometric torus action of rank $r \geq \frac{n}{4} + \log_2 n$, then $M$ is homeomorphic to $\s^n$ or tangentially homotopy equivalent to $\C\pp^{n/2}$.
\end{main}

Theorems \ref{IntroGroveSearle} and \ref{IntroWilkingHomotopy}
show that the answer to our motivating question is ``yes" (at least up to homeomorphism or homotopy) in the case of high enough symmetry rank. On the other hand, our results are slightly weaker than the results in the unweighted setting. We discuss this further in Sections \ref{sec:TorusActions} and \ref{sec:FutureDirections}.

There are two key tools used in the proofs of Theorems \ref{IntroGroveSearle} and \ref{IntroWilkingHomotopy}. The first is an extension of Berger's theorem (Corollary \ref{cor:Berger}) to the weighted case. The proof follows as in \cite{GroveSearle94} and makes use of the O'Neill formula in the weighted case (Theorem \ref{thm:submersions}). The second main tool is a generalization of Wilking's connectedness lemma \cite[Theorem 2.1]{Wilking03} to positive weighted sectional curvature (see Theorem \ref{thm:Connectedness}).

\begin{main}[Wilking's connectedness lemma]\label{thm:IntroConnectedness}
Let $(M^n,g,X)$ be closed with positive weighted sectional curvature.
	\begin{enumerate}
	\item If $X$ is tangent to $N^{n-k}$, a closed, totally geodesic, embedded submanifold of $M$, then the inclusion $N \to M$ is $(n-2k+1)$--connected.
	\item If $X$ and $N^{n-k}$ are as above, and if $G$ acts isometrically on $M$, fixes $N$ pointwise, and has principal orbits of dimension $\delta$, then the inclusion $N \to M$ is $(n-2k+1+\delta)$--connected.
	\item If $X$ is tangent to $N_1^{n-k_1}$ and $N_2^{n-k_2}$, a pair of closed, totally geodesic, embedded submanifolds with $k_1 \leq k_2$, then $N_1 \cap N_2 \to N_2$ is $(n-k_1-k_2)$--connected.
	\end{enumerate}
\end{main}

The only assumption in 
Theorem \ref{thm:IntroConnectedness}
not needed in the unweighted version is that $X$ be tangent to the submanifolds. This of course is true in the unweighted setting where $X=0$.  In the applications, this extra assumption holds since the submanifolds we apply the result to will be fixed-point sets of isometries and $X$ will be invariant under these actions (see Corollary \ref{cor:FrankelGroupAction} and the following discussion). The proof of 
Theorem \ref{thm:IntroConnectedness}
follows from Wilking's arguments in \cite{Wilking03} using the second variation formula for the weighted curvatures derived in \cite{Wylie-pre} in place of the classical one.  

\smallskip

This paper is organized as follows. In Sections \ref{sec:Preliminaries} and \ref{sec:Examples}, we recall the notion of weighted sectional curvature from \cite{Wylie-pre}, define positive weighted sectional curvature, survey its basic properties (including Theorem \ref{thm:LowDimIntro}), and construct a number of examples. In Sections \ref{sec:Averaging}--\ref{sec:Frankel}, we establish these properties and use them to prove Theorem \ref{thm:PWSChomogeneous} as well as generalizations of the O'Neill formulas, Weinstein's theorem, and Wilking's connectedness lemma (Theorem \ref{thm:IntroConnectedness}). In Section \ref{sec:TorusActions}, we use these tools to prove 
Theorems \ref{IntroGroveSearle} and \ref{IntroWilkingHomotopy}. 
In Section \ref{sec:FutureDirections}, we discuss future directions. 

\subsection*{Acknowledgements} We would like to thank Karsten Grove, Frank Morgan, Guofang Wei, Dmytro Yeroshkin, and Wolfgang Ziller for helpful suggestions and discussions. The first author is partially supported by NSF grants DMS-1045292 and DMS-1404670.

\bigskip
\section{Definitions and Motivation} \label{sec:Preliminaries}
\bigskip

In this section, we fix some notation and go into more detail about the motivation for the definition of positive weighted sectional curvature. At the end of this section (see Subsection \ref{sec:symsec}), we address the fact that weighted sectional curvature is not simply a function of $2$--planes in the way that sectional curvature is, and we discuss a symmetrized version of weighted sectional curvature which is.

\subsection{Definition of positive weighted sectional curvature} First we recall some notation from \cite{Wylie-pre}.  For a Riemannian manifold $(M,g)$ and a vector $V$ on $M$, we will call the symmetric $(1,1)$--tensor $R^V$, given by
 	\[R^V(U) = R(U, V)V = \nabla_U \nabla_V V - \nabla_V \nabla_U V - \nabla_{[U,V]} V,\] 
the \textit{directional curvature operator in the direction of $V$}. Given a smooth vector field $X$, the \textit{weighted directional curvature operator in the direction of $V$} is another symmetric $(1,1)$--tensor, 
	\[R_X^V = R^V + \frac 1 2 (L_Xg)(V, V) \id,\]
where $\id$ is the identity operator. The \textit{strongly weighted directional curvature operator in the direction of $V$} is defined as
	\[\overline R_X^V = R^V_X + g(X, V)^2 \id.\]
	
	Given an orthonormal pair $(U,V)$ of vectors in $T_pM$ for some $p \in M$, the sectional curvature $\sec(U,V)$ of the plane spanned by $U$ and $V$ is, by definition, $\sec(U,V) = g(R^V(U),U)$. In the weighted cases, we similarly define
	\begin{eqnarray*}
	\sec^V_X(U) &=& g(R^V_X(U), U) = \sec(V,U) + \frac{1}{2}(L_Xg)(V,V) ,\\
	\overline\sec^V_X(U) &=& g(\overline R^V_X(U), U) = \sec^V_X(U) + g(X,V)^2.
	\end{eqnarray*}
We say that $\sec_X \geq \lambda$ if $\sec_X^V(U) \geq \lambda$ for every orthonormal pair $(V,U)$, or equivalently if all of the eigenvalues of $R_X^V$ are at least $\lambda$ for every unit vector $V$. We define the condition $\overline{\sec}_X \geq \lambda$ in the analogous way. Note that $ \overline{\sec}_X^V(U) \geq \sec_X^V (U)$, so that $\sec_X \geq \lambda$ implies $\overline{\sec}_X \geq \lambda$. 

In terms of this notation we can then rephrase the definition of positive weighted sectional curvature. 

 \begin{nonumberdefinition}
A Riemannian manifold $(M,g)$ equipped with a vector field $X$ has \textit{positive weighted sectional curvature} if
	\begin{itemize}
	\item $\sec_X > 0$, or
	\item $X = \nabla f$ and $\overline{\sec}_f > 0$ for some function $f$.
	\end{itemize}
\end{nonumberdefinition}
Note that, unlike $\sec(U,V)$, the weighted sectional curvatures are not symmetric in $U$ and $V$. This may at first seem unnatural, but it is necessary if we want the weighted sectional curvatures to agree with the Bakry--Emery Ricci curvatures in dimension $2$ as the Bakry--Emery Ricci tensors of a surface with density will generally have two different eigenvalues. See Section \ref{sec:symsec} for a discussion of a symmetrized version.

Also note that $\sec^V_X$ and $\overline\sec^V_X$ average to Bakry--Emery Ricci curvatures in the following sense. Let $\{E_i\}_{i=1}^{n-1}$ be an orthonormal basis of the orthogonal complement of $V$, then
\begin{eqnarray}
\Ric_{(n-1) X} (V,V) &=& \sum_{i=1}^{n-1} \sec^V_X(E_i) \label{eqn:AvgSec1} \\
\Ric_{(n-1)X}^{-(n-1)}(V,V) &=& \sum_{i=1}^{n-1} \overline{\sec}^V_X(E_i) \label{eqn:AvgSec2}
\end{eqnarray}
In particular, for surfaces, $\sec_X \geq \lambda$ is equivalent to $\Ric_X \geq \lambda $ and similarly for $\overline{\sec}_X$ and $\Ric_X^{-1}$. The curvature $\Ric_{(n-1)X}^{-(n-1)}$ is an example of Bakry--Emery Ricci curvature with negative $m$ which has been studied recently in \cite{ KolesnikovMilman, Ohta}.

\subsection{Properties of positive weighted sectional curvature}
Now that we have introduced the main equations involving weighted sectional curvature, we summarize some of  properties that the condition of positive weighted sectional curvature shares with positive sectional curvature. We then give a basic outline of how these facts lead to the  proof of 
Theorems \ref{IntroGroveSearle} and \ref{IntroWilkingHomotopy}. 

First, positive weighted sectional curvature is preserved under covering maps. Namely if $(M,g,X)$ has positive weighted sectional curvature and $\tilde{M}$ is a cover of $M$, then $(\tilde{M}, \tilde{g}, \tilde{X})$ has positive weighted sectional curvature where $\tilde g$ and $\tilde X$ are the pullbacks of $g$ and $X$ respectively under the covering map. 

A second property of positive weighted sectional curvature is that the fundamental group is finite in the compact case. Indeed, this follows from \cite[Theorem 1.1]{Wylie08} and \cite[Theorem 1.14]{Wylie-pre} by using the fact that positive weighted sectional curvature lifts to covers:
 \begin{theorem} \label{thm:pi1finite}
Let $(M,g)$ be a complete Riemannian manifold. 
\begin{itemize}
\item If there exists a vector field $X$ such that $\mathrm{Ric}_X > \la > 0$,  or 
\item if $M$ is compact and there is a function $f$ such that $\mathrm{Ric}^{-(n-1)}_f > \la >0$, 
\end{itemize}
then $\pi_1(M)$ is finite. 
\end{theorem} 

This theorem immediately implies the classification of compact $2$-- and $3$--dimensional manifolds with positive weighted sectional curvature stated in Theorem \ref{thm:LowDimIntro}. Indeed, this follows in dimension two from the classification of surfaces and in dimension three from the Ricci flow proof of the Poincar\'e conjecture.

We remark that, for positive Ricci curvature, the finiteness of fundamental group follows from the Bonnet--Myers' diameter estimate. There is no diameter estimate for the weighted curvatures as there are complete non-compact examples with $\sec_f > \la > 0$ (see Example \ref{Ex:Gaussian}).

A third property of positive weighted sectional curvature is that the vector field $X$ can always be chosen so that it is invariant under a fixed compact group of isometries. We interpret this as a shared property with positive sectional curvature since the zero vector field is always invariant. Specifically we have:

\begin{corollary}\label{cor:PWSCaveraging}
If $(M,g,X)$ has positive weighted sectional curvature, and if $G$ is a compact subgroup of the isometry group of $(M,g)$, then $X$ can be replaced by a $G$--invariant vector field $\tilde X$ such that $(M,g,\tilde X)$ has positive weighted sectional curvature.
\end{corollary} 

When $M$ is compact, the isometry group is compact, hence this corollary applies in this case where $G$ is the whole isometry group. As we mentioned in the introduction, reducing to the invariant case will be key in most of our results. Corollary \ref{cor:PWSCaveraging} follows immediately from Lemmas \ref{lem:averaging} and \ref{lem:u-averaging} below.

A fourth property of positive weighted sectional curvature is that Riemannian submersions preserve it in the following sense:

\begin{corollary} \label{cor:PWSCSubmersion}
Let $\pi:(M,g) \to (B,h)$ be a Riemannian submersion. Let $X$ be a vector field $X$ on $M$ that descends to a well defined vector field $\pi_*X$ on $B$. If $(M,g,X)$ has positive weighted sectional curvature, then so does $(B,h,\pi_*X)$.\end{corollary}

This follows immediately from a generalization of O'Neill's formulas proved below (Theorem \ref{thm:submersions}). We also obtain from O'Neill's formulas that Cheeger deformations preserve positive weighted sectional curvature (Lemma \ref{Lemma:CheegerDeformation}). 

Corollary \ref{cor:PWSCSubmersion} implies the following: If $(M,g,X)$ is compact with positive weighted sectional curvature and $G$ is a closed subgroup of the isometry group that acts freely on $M$, then $M/G$ admits positive weighted sectional curvature. Indeed, by Corollary \ref{cor:PWSCaveraging} we can modify $X$ so that it is $G$--invariant and so descends to a vector field on $M/G$ via the quotient map $\pi:M \to M/G$. It follows that $M/G$ equipped with the vector field $\pi_*X$ has positive weighted sectional curvature by Corollary \ref{cor:PWSCSubmersion}. We implicitly use this fact in the proof of Berger's theorem (see Corollary \ref{cor:Berger}).

Finally, a crucial property of positive weighted sectional curvature is that Synge-type arguments for positive sectional curvature generalize to the weighted setting. This follows from studying a second variation formula for energy of geodesics that was derived in \cite{Wylie-pre}. Given a variation $\overline \gamma:[a,b]\times(-\ep,\ep) \to M$ of a geodesic $\gamma = \overline\gamma(\cdot,0)$, let $V = \left.\frac{\partial\overline\gamma}{\partial s}\right|_{s=0}$ denote the variation vector field along $\gamma$. The second variation of energy is given by
	\[\left.\frac{d^2}{ds^2}\right|_{t=0} E(\gamma_s)
	= I(V,V) + \left.g\of{\frac{\partial^2\overline\gamma}{\partial s^2}}\right|_{t=a}^{t=b},\]
where $I(V,V)$ is the index form of $\gamma$.  The usual formula for the index form is 
\[ I(V,V) = \int_a^b \of{ |V'|^2 -R^{\gamma'}(V,V) } dt .\]
 In terms of the weighted directional curvature operators, the index form can be re-written as follows (see \cite[Section 5]{Wylie-pre}):
	\begin{eqnarray}
	\hspace{.1in}I(V,V)\hspace{-.1in} &=&\hspace{-.1in} \int_a^b \of{ |V'|^2 
	- R_X^{\gamma'}(V,V)
	- 2g(\gamma', X)g(V,V')}dt
	+ \left.g\of{\gamma', X} |V|^2\right|_{t=a}^{t=b}\label{eqn:IndexForm1}\\
	~&=&\hspace{-.1in} \int_a^b \of{ |V' - g(\gamma',X)V|^2 
	- \overline R_X^{\gamma'}(V, V)}dt
	+ \left.g\of{\gamma', X} |V|^2\right|_{t=a}^{t=b}\label{eqn:IndexForm2}
	\end{eqnarray}
It may not be immediately apparent why these formulas are natural, but they do allow us to generalize Synge-type arguments using the following. 
\begin{lemma}\label{lem:SecondVariation}
Fix a triple $(M,g,X)$. Let $\gamma:[a,b] \to M$ be a geodesic on $M$, and let $Y$ be a unit-length, parallel vector field along and orthogonal to $\gamma$.
	\begin{enumerate}
	\item If $\sec_X > 0$, then the variation $\gamma_s(t) = \exp(sY)$ of $\gamma$ satisfies
	\[\left.\frac{d^2}{ds^2}\right|_{s=0}E(\gamma_s) 
	< \left.g\of{\gamma'(t),X_{\gamma(t)}}\right|_{t=a}^{t=b}.\]
	\item If $X = \nabla f$ and $\overline\sec_f > 0$, then the variation $\gamma_s(t) = \exp(s e^f Y)$ of $\gamma$ satisfies
	\[\left.\frac{d^2}{ds^2}\right|_{s=0}E(\gamma_s) 
	< \left. e^{f(\gamma(t))}g\of{\gamma'(t),X_{\gamma(t)}}\right|_{t=a}^{t=b}.\]
	\end{enumerate}
\end{lemma}

This lemma is used in Sections \ref{sec:Weinstein} and \ref{sec:Frankel} to generalize theorems of Weinstein, Berger, Synge, and Frankel, as well as Wilking's connectedness lemma. Once we have these results, it is not hard to see the how to generalize the proofs of 
Theorems \ref{IntroGroveSearle} and \ref{IntroWilkingHomotopy} to the weighted setting. We indicate briefly how the arguments go.

The proofs proceed by induction on the dimension $n$, the base cases $n \in \{2,3\}$ being handled by the classification of simply connected, compact manifolds in these dimensions. If a torus acts effectively on $M$, which has positive weighted sectional curvature, then we obtain a fixed point set $N$ of lower dimension by Berger's theorem. The fixed point set of a subgroup of isometries is always a totally geodesic submanifold, and since we can assume $X$ is invariant under the group, we also obtain that $X$ is tangent to $N$. It follows immediately that $N$ with restricted vector field $X$ also has positive weighted sectional curvature. Finally, the torus action restricts to $N$, so it might follow by induction on the dimension of the manifold that $N$ satisfies the conclusion of the theorem. If the induction hypothesis does not apply, the codimension of $N$ is small and other arguments are used to again show that $N$ satisfies the conclusion of the theorem. By applying Wilking's connectedness lemma, the topology of $M$ is recovered from the topology of $N$.

\subsection{Symmetrized weighted sectional curvature}\label{sec:symsec}

Unlike sectional curvature, a weighted sectional curvature $\sec_X$ on a triple $(M,g,X)$ is not a function of $2$--planes. In this subsection, we define a symmetrized version of this quantity that is. We also compare the notions of sectional curvature and symmetrized weighted sectional curvature.

Given a vector field $X$ on a Riemannian manifold $M$,
$\sec_X$ can be regarded as a function $\sec_X(\sigma,V)$ of $(\sigma, V)$, where $\sigma \subseteq T_p M$ is a $2$--plane and $V$ is a unit vector in $\sigma$. To evaluate $\sec_X(\sigma,V)$, choose either of the two unit vectors in $\sigma$ orthogonal to $V$, call it $U$, and evaluate $\sec_X^V(U)$. 

Note that the unit circle $\s^1(\sigma)$ in $\sigma$ is defined by the metric, so it makes sense to average over unit vectors $e^{i\theta} \leftrightarrow V \in \s^1(\sigma)$. We denote this by
\[\sym\sec_X(\sigma) = \frac{1}{2\pi} \int_0^{2\pi} \sec_X(\sigma,e^{i\theta}) d\theta.\]

One can similarly define $\sym\overline\sec_X$. One appealing aspect of this curvature quantity is that it is the same kind of object as $\sec$, a function on two--planes.

This definition was motivated by a suggestion of Guofang Wei, who suggested looking at the quantity
	\[\sec_X^V(U) + \sec_X^U(V).\]
Note that $\frac 1 2 \of{\sec_X^V(U) + \sec_X^U(V)}$ equals $\sym\sec_X$ and likewise in the strongly weighted case. 

We analyze the conditions $\sym\sec_X > 0$ and $\sym\overline\sec_X > 0$ in dimension two. First, it is clear that in any dimension
	\[
	\begin{array}{ccc}
	\sec_X > 0	& \Rightarrow 		& \sym\sec_X > 0\\
	\Downarrow	&	~			& \Downarrow\\
	\overline\sec_X > 0	&\Rightarrow	& \sym\overline\sec_X > 0
	\end{array}
	\]

Second, in dimension 2, $\sym \sec_f = \frac{\scal}{2} + \Delta f$. This is the same as the weighted Gauss curvature studied in \cite{CorwinHoffmanHurderSesumXu06, CorwinMorgan11}, which contain proofs that the Gauss--Bonnet theorems hold for this weighted curvature. In particular, we have the following (compare \cite[Proposition 5.3]{CorwinHoffmanHurderSesumXu06}):

\begin{theorem}[Gauss--Bonnet]\label{thm:GBsymsec}
	If $M^2$ is orientable, then $\int_M \sym\sec_X = 2\pi\chi(M)$.
	\end{theorem}

This gives the following generalization of one case of Theorem \ref{thm:LowDimIntro}, which implies that a $2$--dimensional, compact manifold $M$ that admits $\sec_X > 0$ for some vector field $X$ is diffeomorphic to a spherical space form.. 

\begin{corollary} If $M^2$ is compact and admits a metric and vector field $X$ with $\sym \sec_X >0$, then $M^2$ is diffeomorphic to a spherical space form. \end{corollary}

On the other hand, the torus $T^2$, while it does not admit $\sym \sec_X > 0$, does admit a metric with $\sym \overline \sec_X > 0$. To see this equip the torus with a flat metric and a unit-length Killing field $X$, then we have 
	\[\sym\overline\sec_X = 0 + 0 + \frac{1}{2\pi}\int_0^{2\pi}g\of{X,e^{i\theta}}^2d\theta = 1.\]
In fact, this example immediately generalizes as follows:

\begin{proposition}
If $(N,g)$ is a Riemannian manifold with positive sectional curvature, then $\s^1 \times N$ admits a metric and a vector field $X$ such that $\sym \overline\sec_X > 0$.
\end{proposition}

\begin{proof}
Let $g$ be the product metric, and let $X$ denote the unit-length Killing field tangent to the circle factor. If $\sigma$ is a two-plane tangent to $N$, then
	\[\sym\overline\sec_X(\sigma) \geq \sec^{g_N}(\sigma) > 0.\]
If $\sigma$ is a two-plane not contained in the tangent space to $N$, then
	\[\sym\overline\sec_X(\sigma) \geq \frac 1 2 |\mathrm{proj}_\sigma(X)|^2 > 0,\]
where $\mathrm{proj}_\sigma$ denotes the projection onto $\sigma$.
\end{proof}

This raises the following question. 

\begin{question} Does the torus admit $\overline{\sec}_X >0$ or $\sym \overline\sec_f > 0$? More generally are there compact manifolds with $\overline{\sec}_X >0$ or $\sym \overline\sec_f > 0$ and infinite fundamental group? 
\end{question}

We point out that Gauss--Bonnet type arguments do not seem to give a different proof that any compact surface with density with $\overline{\sec}_f>0$ is a sphere. Indeed, if we trace $\overline{\sec}_X$, we obtain 
\[ \scal + \mathrm{div}(X) + |X|^2. \]
The integral of this is $4 \pi \chi(M) + \int_M |X|^2 d \vol_g$, which is not a topological quantity. 

We also note that the Gauss-Bonnet theorem gives interesting information about other inequalities involving curvature. First we consider a positive lower bound. 

\begin{proposition} \label{pro:SA} Let $(M,g)$ be a compact surface with $\sym \sec_X \geq 1$. The area of $M$ is at most $4\pi$. Moreover, if $\sec_X \geq 1$ and the area of $M$ is $4\pi$ then $(M,g)$ is the round sphere and $X$ is a Killing field. 
\end{proposition}

\begin{proof} We apply the discussion above to the universal cover $\tilde M$ of $M$, endowed with the pulled back metric $\tilde g$ and vector field $\tilde X$. It follows that $\chi(\tilde M) > 0$, so that $\chi(\tilde M) = 2$ and $\area(M) \leq \area(\tilde M) \leq 4 \pi$. Moreover, if $\area(M) = 4\pi$, then both of these inequalities are equalities. In particular, $\pi_1(M)$ is trivial and $\sec_{\tilde X} = \Ric_{\tilde X}= \tilde g$, so that $(M,g,X) = (\tilde M,\tilde g, \tilde X)$ is a compact, two-dimensional Ricci soliton. A result of Chen, Lu, and Tian \cite{ChenLuTian06} then shows that $M$ has constant curvature $1$ and that $X$ is a Killing field. Since $M$ is simply connected, this proves the proposition.
\end{proof}

We can also consider the case of negative curvature in dimension $2$. It was shown in \cite{Wylie-pre} that if a compact manifold has $\overline{\sec}_X \leq 0$ then the universal cover is diffeomorphic to Euclidean space, showing that a compact surface admits $\overline{\sec}_X \leq 0$ if and only if it is not the sphere or real projective space. In fact, the Gauss--Bonnet argument improves this result for surfaces as it shows that the conclusion holds if $\sym \sec_X \leq 0$. Moreover, it also shows that if a metric on the torus supports a vector field with $\sym \sec_X \leq 0$ then the metric is flat and $X$ is Killing. In particular the 2-torus has no metric with density on it with $\sym \sec_X < 0$. 

The discussion above, along with the work of Corwin and Morgan \cite{CorwinMorgan11} certainly shows that the study of the symmetrized weighted sectional curvature is warranted. In fact, the results in Section \ref{sec:Submersions} of this paper about Riemannian submersions and Cheeger deformation have analogues for the symmetrized curvatures with the same proofs. On the other hand, there does not seem to be a good second variation formula for the symmetrized curvatures which can give us a version of Lemma \ref{lem:SecondVariation}.   Note that the unsymmetrized curvatures also appear in the second variation of the weighted distance, see \cite{Morgan06, Morgan09}.  Without some kind of second variation formula for the symmetrized curvatures, it seems unlikely that the other results of this paper can be generalized to the symmetrized case or that many of the facts for surfaces mentioned above can be generalized to higher dimensions.

\bigskip
\section{Examples}\label{sec:Examples}
\bigskip

In this section, we discuss a number of examples of metrics with positive weighted curvature, including some which do not have positive sectional curvature. As a warm-up we first consider the case of products. 

\begin{definition} Given $(M_1, g_1, X_1)$ and $(M_2, g_2, X_2)$ where $(M_i, g_i)$ are Riemannian manifolds and $X_i$ are smooth vector fields, the product of $(M_1, g_1, X_1)$ and $(M_2, g_2, X_2)$ is the triple $(M_1 \times M_2, g_1 + g_2, X_1 + X_2)$. 
\end{definition}

A basic fact about positive sectional curvature is that it is not preserved by taking products, as the sectional curvature of a plane spanned by vectors in each factor is zero. Indeed, one of the most famous open problems in Riemannian geometry is the Hopf conjecture which states that $\s^2 \times \s^2$ does not admit any metric of positive sectional curvature.   In the weighted case, there are noncompact examples of  products which have positive weighted sectional curvature.

\begin{example} \label{Ex:Gaussian} We define the $1$-dimensional Gaussian as the real line $\mathbb{R}$ with coordinate $x$, standard metric $g = dx^2$, and vector field $X = \frac{1}{2}\nabla( x^2) = x \frac{d}{dx}$. This triple has $\sec_X = 1$. If we take the product of two $1$-dimensional Gaussians we obtain a $2$-dimensional Gaussian. That is, we obtain $\mathbb{R}^2$ with the Euclidean metric and vector field $X = \nabla f$ where $f(A) = \frac{1}{2} |A|^2$. This triple still has $\sec_X = 1$. Moreover, taking further products we obtain the $n$-dimensional Gaussian as the product of $n$ one-dimensional Gaussians all of which have $\sec_X = 1. $
\end{example}

On the other hand, it is easy to see that such examples cannot exist in the compact case. 

\begin{proposition}
No product of the form $(M_1 \times M_2, g_1+g_2, X_1+X_2)$ with one of the $M_i$ compact has positive weighted sectional curvature. 
\end{proposition}
\begin{proof}
Let $M_1$ be the compact factor and suppose first that $\sec_X > 0$. Consider the ``vertizontal" curvature given by $Y$ tangent to $M_1$ and $U$ tangent to $M_2$, 
\[ \sec_X (Y, U) = \sec(Y,U) + \frac{1}{2} L_X g(Y,Y) = \frac{1}{2} L_{X_1} g_1(Y,Y) \]
This shows that if $\sec_X > 0$, then $\frac{1}{2} L_{X_1} g_1 > 0$. This is impossible if $M_1$ is compact by the divergence theorem as $\mathrm{tr}\left(L_{X_1} g\right) = \mathrm{div}(X_1)$. 

The case where $\overline{\sec}_f > 0$ is analogous.  In that case we obtain that the function $u_1 = e^{f_1}$ has $\mathrm{Hess}_{g_1} u_1 > 0$ on $M_1$,  which is again impossible on a compact manifold. 
\end{proof}

This shows that the Hopf conjecture is also an interesting for weighted sectional curvature.

\begin{question}[Weighted Hopf conjecture] Does $\s^2 \times \s^2$ admit a metric and vector field with positive weighted sectional curvature? 
\end{question}

In the next few sections, we investigate examples with positive weighted sectional curvature using the simple construction of warped products over a one-dimensional base. As we can see even in the case of products, it is easier to construct non-compact examples than compact ones, so we will investigate the non-compact case first. 

\subsection{Noncompact Examples}

A warped product metric over a $1$--dimensional base is a metric of the form $g = dr^2 + \phi^2(r) g_N$ where $N$ is an $(n-1)$--dimensional manifold. Up to rescaling $\phi$ and the fiber metric $g_N$ and re-parametrizing $r$ there are three possibilities for the topology of complete metrics of this form: 
\begin{enumerate}
\item If $\phi(r) >0$ for $r \in \mathbb{R}$ and $(N,g_N)$ is complete, then $g$ gives a complete metric on $\mathbb{R} \times N$. If $\phi$ is also periodic, then we can take the quotient a get a metric on $\s^1 \times N$. 
\item If $\phi(r) >0$ for $r\in (0, \infty)$ and $\phi$ is an odd function with $\phi'(0) = 1$, and if $(N,g_N)$ is a round sphere of constant curvature $1$, then $g$ defines a complete rotationally symmetric metric on $\mathbb{R}^n$
\item If $\phi(r)>0$ for $r \in (0,R)$ and $\phi$ is an odd function at $0$ and $R$ with $\phi'(0) = 1$ and $\phi'(R) = -1$, and if $(N,g_N)$ is a round sphere of constant curvature $1$, then $g$ defines a complete rotationally symmetric metric on $\s^n$. 
\end{enumerate}

For $Y,Z$ tangent to $g_N$, we have the following well known formulas for the curvature operator of a one-dimensional warped product,
 \begin{eqnarray*}
 \mathcal{R}(\partial_r \wedge Y) 
 	&=& -\frac{\phi''}{\phi} \partial_r \wedge Y \\
 \mathcal R(Y \wedge Z)
 	&=& \mathcal R^N(Y \wedge Z) - \of{\frac{\phi'}{\phi}}^2 Y \wedge Z
 \end{eqnarray*}
where $\mathcal R^N$ denotes the curvature operator of $N$. We will be interested in lower bounds on weighted curvature of the warped product.  All of our examples will also have the property that $X = \nabla f$, so we focus only on this case. The following lemma simplifies the problem of proving such lower bounds for warped product metrics over a one-dimensional base.

\begin{lemma}\label{lem:SinglyWarpedTest}
Let $dr^2 + \phi^2(r) g_N$ be a warped product metric, and assume $f$ is a smooth function that only depends on $r$. The weighted curvature $\sec_f \geq \lambda$ if and only if
	\begin{eqnarray*}
	\lambda &\leq& \sec_f^{\partial_r}(Y) = -\frac{\phi''}{\phi} + f'',\\
	\lambda &\leq& \sec_f^{Y}(\partial_r) = -\frac{\phi''}{\phi} + f'\frac{\phi'}{\phi},~\mathrm{and}\\
	\lambda &\leq& \sec_f^{Y}(Z)		 = \frac{\sec^{g_N}(Y,Z) - (\phi')^2}{\phi^2} + f'\frac{\phi'}{\phi},
	\end{eqnarray*}
for all orthonormal pairs $(Y,Z)$, where $Y$ and $Z$ are tangent to $N$.

Similarly, $\overline\sec_f \geq \lambda$ if and only if these three inequalities hold with $f'$ replaced by $u'/u$ and $f''$ replaced by $u''/u$, where $u = e^f$.
\end{lemma}

This lemma implies that one can show $\sec_f \geq \la$ for these metrics by plugging in ``test pairs'' of the form $(\partial_r, Y)$, $(Y,\partial_r)$, and $(Y,Z)$, where $Y$ and $Z$ are tangent to $N$. In particular, if $\sec^{g_N}$ is bounded from below, then proving $\sec_f \geq \lambda$ reduces to showing three inequalities involving the functions $\phi$ and $f$.

\begin{proof}
As the proof is similar, we omit the proof in the strongly weighted case. Let $U = a\partial_r + Y$ and $V = b\partial_r + Z$ be an arbitrary orthonormal pair of vectors, where $Y$ and $Z$ are tangent to $N$. By orthonormality,
	$|Y \wedge Z|^2 = 1 - a^2 - b^2$,
so $a^2 + b^2 \leq 1$. Since $\partial_r \wedge (aZ-bY)$ and $Y\wedge Z$ are eigenvalues of the curvature operator, we have
	\begin{eqnarray*}
	\sec(U,V) &=& \inner{\mathcal R(U\wedge V), U \wedge V}\\
			&=& - \frac{\phi''}{\phi} |\partial_r \wedge (aZ - bY)|^2 
			 + \of{\frac{\sec^N(Y,Z) - (\phi')^2}{\phi^2}} |Y\wedge Z|^2\\
			&=& - \frac{\phi''}{\phi} \of{a^2 + b^2}
				+ \of{\frac{\sec^N(Y,Z) - (\phi')^2}{\phi^2}} \of{1 - a^2 - b^2}.
	\end{eqnarray*}
Next, we calculate
	\[\Hess f(U,U)|V|^2 = a^2 f'' + (1-a^2) f' \frac{\phi'}{\phi}.\]
Observe that $\sec_f^U(V) = \sec(U,V) + \Hess f(U,U) |V|^2$ is a linear function in the quantities $a^2$ and $a^2 + b^2$. Moveover, these quantities vary over a triangle since
	\[0 \leq a^2 \leq a^2 + b^2 \leq 1,\]
so the minimal (and maximal) values of $\sec_f^U(V)$ occur at one of the three corners. This proves the lemma since these corners correspond to orthonormal pairs of the form $(\partial_r, Y)$, $(Y,\partial_r)$, and $(Y,Z)$.
\end{proof}

As a first application of Lemma \ref{lem:SinglyWarpedTest} we consider the problem of prescribing positive weighted sectional curvature locally on a subset of the round sphere. 

\begin{proposition} 
Let $M$ be a round sphere of constant curvature $1$ and $H^+$ an open round hemisphere in $M$. For any $\lambda \in \mathbb{R}$, there is a density on $H^+$ with $\sec_f \geq \lambda$ and there is no density defined on an open set containing the closure of $H^+$ with $\overline{\sec}_f > 1$. 
\end{proposition}

\begin{proof}
First we prove the non-existence. It suffices to show that a geodesic ball $B$ of radius $\frac \pi 2 + \ep$ cannot admit a density $f$ such that $\overline\sec_f > 1$. On $B$, we can write the round metric as the warped product $dr^2 + \sin^2(r)g_{\s^{n-1}}$, where $r \in \left(0, \frac{\pi}{2} + \varepsilon\right)$. By Lemma  \ref{lem:u-averaging}  proved in the next section, we can assume that $f =f(r)$. By Lemma \ref{lem:SinglyWarpedTest}, $\overline{\sec}_f > 1$ only if $\frac{u'}{u} \cot(r) > 0$. However, $\cot\left(\frac{\pi}{2} \right) = 0$, so the second inequality is impossible to satisfy. 
 
On the other hand, in order to find a density $f$ with $\sec_f \geq \lambda$, we only need that
 	\begin{eqnarray*}
 	f'' &\geq& \lambda -1 \qquad \text{and} \qquad 
 	f' \cot(r) \geq \lambda-1 
 	\end{eqnarray*}
Such a density exists, e.g., $f$ given by
\[ f(r) = (\lambda-1) \int \tan(x)dx = -(\lambda-1)\log(\cos(r)). \]
satisfies these properties. Note that in these examples, $f$ blows up at the equator $r = \frac{\pi}{2}$.  
\end{proof}

On the other hand, we note the general fact that every point $p$ in a Riemannian manifold has a neighborhood $U$ supporting a density such that $\sec_f \geq \lambda$. 

\begin{proposition} Let $(M,g)$ be a Riemannian manifold, $p \in M$, and $\lambda \in \mathbb{R}$. There is an open set $U$ containing $p$ which supports a density $f$ such that $\sec_f \geq \lambda$ on $U$. \end{proposition}

\begin{proof} 
 Let $r$ be the distance function to $p$. Since $\Hess r \sim 1/r$ as $r \rightarrow 0$, there exists $0<\varepsilon<1$ such that $r \Hess r > \ep g$ on $B(p,\varepsilon)$. Let $\rho = \inf \sec(B(p,\varepsilon), g)$. Define $f = \frac{\lambda-\rho}{2\varepsilon} r^2$ . We have that 
\[ \Hess f = \frac{\lambda-\rho}{\varepsilon} dr \otimes dr + \frac{\lambda-\rho}{\varepsilon}r \Hess r \geq (\lambda-\rho) g, \]
which implies that $\sec_f \geq \lambda$ on $B(p, \varepsilon)$. 
\end{proof}

Now we come to our first complete example. 

\begin{proposition} \label{pro:WPPosCurv}
Let $(N,g_N)$ be a metric of non-negative sectional curvtaure. For any $\lambda$, the metric $g = dr^2 + e^{2r} g_N$ on $\mathbb{R} \times N$ admits a density of the form $f=f(r)$ such that $\overline{\sec}_f \geq \lambda$. On the other hand, $g$ admits no density of the form $f=f(r)$ with $\sec_f \geq -1 + \ep$ with $\ep > 0$. 
\end{proposition}

\begin{proof} 
Set $\phi(r) = e^r$. Because $N$ has non-negative sectional curvature, Lemma \ref{lem:SinglyWarpedTest} implies $\overline\sec_f \geq \lambda$ if and only if $-1 + \frac{u''}{u} \geq \lambda$ and $-1 + \frac{u'}{u} \geq \lambda$. This can be achieved by taking $u = e^{Ar}$ for some sufficiently large $A \in \R$.

On the other hand, for a general $f= f(r)$, if we have $\sec_f \geq -1 + \ep$, then $f$ satisfies $f''(r)\geq \ep$ and $f'(r) \geq \ep$ for all $r \in \R$. This is impossible.
\end{proof} 

\begin{remark} Gromoll and Meyer \cite{GromollMeyer69} proved that a non-compact, complete manifold with $\sec>0$ is diffeomorphic to Euclidean space. These examples show this is not true for $\overline{\sec}_f>0$. Moreover, Cheeger and Gromoll \cite{CheegerGromoll72} showed that a non-compact complete manifold with $\sec \geq 0$ is the normal bundle over a compact totally geodesic submanifold called a soul. While our examples are topologically $\mathbb{R} \times N$, we note that the cross sections $\{r_0\} \times N$ are not geometrically a ``soul" as they are not totally geodesic. 
\end{remark}

\begin{remark}
If we take $g_N$ to be a flat metric, then the metric $g = dr^2 + e^{2r} g_N$ is a hyperbolic metric. If we also choose $f(r) = r$, then we get a density with constant $\overline{\sec}_f = 0$. 
\end{remark}

\subsection{Compact Examples}

Now we give examples of rotationally symmetric metrics on the $n$--sphere which admit a density $f$ such that $\sec_f>0$ but do not have $\sec\geq 0$.  

In general, a rotationally symmetric on the sphere will be of the form $g = dr^2 + \phi^2(r)g_{\s^{n-1}}$ for $r\in [0,2L]$. The smoothness conditions for the warping function $\phi$ and density function $f$ are that $\phi(0) = \phi(2L) = 0$, $\phi'(0) = 1$, $\phi'(2L) = -1$, $\phi^{(even)}(0)= \phi^{(even)}(2L) = 0$ and $f'(0) = f'(2L) = 0$. Our main construction is contained in the following proposition. 

\begin{proposition}\label{pro:RotSym}
There are rotationally symmetric metrics on $\s^n$ which support a density $f$ such that $\sec_f >0$, but which do not have $\sec \geq 0$. 
\end{proposition}

\begin{proof} 

First we define $\phi(r) = r$ on $[0, \pi/6]$ and $\phi(r) = \sin(r)$ on $[\pi/3, \pi/2]$. On the interval $(\pi/6, \pi/3)$, extend $\phi$ smoothly so that $\phi'' \leq 0$ and $\phi' \geq 0$.  Then we reflect $\phi$ across $\pi/2$ to obtain a warping function defined on $[0, \pi]$ that gives a smooth rotationally symmetric metric on the sphere. Geometrically, this metric consists of two flat discs connected by a region of positive curvature.

Now define $f(r) = \frac{1 }{2} r^2$ on $[0, \pi/3]$ and extend $f$ on $(\pi/3, \pi/2]$ so that $f'>0$ on $(\pi/3, \pi/2)$ and has $f^{(odd)}(\pi/2) = 0$ so that $f$ also defines a smooth function when reflected across $\pi/2$.   

Now we consider the potential function $\lambda f''$ for a positive constant $\lambda.$ The table below shows the values for the eigenvalues of the curvature operator and Hessian of $\lambda f$ on the different regions

\[\begin{array}{|c|c|c|c|c|}\hline
~	& \frac{-\phi''}{\phi}  	& \frac{1 - (\phi')^2}{\phi^2} & \lambda f'' & \lambda f' \frac{\phi'}{\phi} \\\hline
[0, \pi/6] 		& 0 		& 0 		&	\lambda 	& \lambda \\\hline
(\pi/6, \pi/3] 	& > 0	& > 0	& \lambda 	& \geq\lambda \\\hline
(\pi/3, \pi/2] 	& 1 		& 1 		& \lambda f'' 	& \geq 0 \\\hline
\end{array}\]
By Lemma \ref{lem:SinglyWarpedTest}, $\sec_f \geq \lambda$ on $[0, \pi/3]$. On $(\pi/3, \pi/2]$ note that $f'' < 0$ somewhere since $f'$ must decrease from $\pi/3$ to $0$. However, by choosing $\lambda$ small enough we can make $1+\la f'' \geq \la$ on $[\pi/3, \pi/2]$, and then we will have $\sec_f \geq \la$ everywhere. 

We have thus constructed examples with $\sec_f > 0$ but which do not have $\sec>0$. Of course, this example does have $\sec \geq 0$. However, since having $\sec_f>0$ is an open condition we can perturb the metric in an arbitrary small way and still have $\sec_f >0$. This will give metrics with some negative sectional curvatures which still have $\sec_f > 0$. 
\end{proof}

On the other hand, we note that most rotationally symmetric metrics on the sphere do not have any density such that $\overline{\sec}_f > 0$. 

\begin{proposition} \label{pro:RotObstruct} Let $g = dr^2 + \phi^2(r) g_{\s^n-1}$, $r \in [0,2L]$, be a metric on $\s^n$
\begin{enumerate}
\item If there is a density $f$ such that $\sec_f > 0$ then $\int_0^{2L} \frac{-\phi''(r)}{\phi(r)} dr \geq 0 $. 
\item If there is a density $f$ such that $\overline{\sec}_f > 0$, then $\phi$ has a unique critical point $t_0$. Moreover, at $t_0$, the metric has positive sectional curvature.
\end{enumerate}
\end{proposition}

\begin{proof}
By Lemmas \ref{lem:averaging} and  \ref{lem:u-averaging} we can assume in either case that $f$ is a function of $r$. Both results are simple consequences of the equations for curvature. For the first, we consider the equation 
\[ \sec_f^{\partial_r} (Y) = \frac{-\phi''}{\phi} + f'' > 0. \]
For $f$ to define a smooth function, we must have $f'(0) = f'(2L) = 0$, so integrating the equation gives (1).  In dimension $2$, this is the Gauss--Bonnet Theorem which we will discuss in the Appendix, see Theorem \ref{thm:GBsymsec}. 

For (2), consider a point where $\phi'(t) = 0$. Fix an orthonormal pair of vectors, $Y$ and $Z$, at this point that are tangent to $N$. Since $\Hess u(Y,Y) = u' \frac{\phi'}{\phi} g(Y,Y) = 0$, the only way $\overline{\sec}_f^Y(\partial_r)$ and $\overline\sec_f^Y(Z)$ can be positive is if $\sec(\partial_r, Y)$ and $\sec(Y,Z)$ are positive. It follows that all sectional curvatures are positive at this point. Moreover, it follows that $\phi''(t) < 0$ at each critical point $t$, so there can be at most one critical point of $\phi$.
\end{proof} 

\begin{remark} Proposition \ref{pro:RotObstruct}, part (2), shows that a spherical ``dumbell" metric consisting of two spheres connected by a long neck of non-positive curvature does not have any density with $\overline{\sec}_f > 0$.
\end{remark}

Now we consider doubly warped products of the form 
	\[ g = dr^2 + \phi^2(r) g_{\s^k} 
			 + \psi^2(r)g_{\s^m} \qquad r \in [0,L]. \] 
These metrics are also cohomogeneity one with $G = \Or(k+1) \times \Or(m+1)$, so by Lemmas  \ref{lem:averaging} and  \ref{lem:u-averaging}  we can assume that  the density is of the form $f=f(r)$. We also have 
	\[ \Hess r = \phi' \phi g_{\s^{k}} + \psi' \psi g_{\s^m}. \]
So 
	\[ \Hess f = f'' dr^2 + f' \phi' \phi g_{\s^{k}} 
					 + f' \psi' \psi g_{\s^m}.\]
In order for $f$ to be $C^2$ we thus need $f'(0) = f'(L) = 0$. 

We let $Y,Z$ denote vectors in the $\s^k$ factor and $U,V$ be vectors in the $\s^m$ factor.  The curvature operator in this case is 
 \begin{eqnarray*}
 \mathcal{R}(\partial_r \wedge Y) &=& -\frac{\phi''}{\phi} \partial_r \wedge Y \\
  \mathcal{R}(\partial_r \wedge U) &=& -\frac{\psi''}{\psi} \partial_r \wedge U \\
  \mathcal{R}(Y \wedge Z) &=& \frac{1 - (\phi')^2}{\phi^2} Y \wedge Z\\
   \mathcal{R}(U \wedge V) &=& \frac{1 - (\psi')^2}{\psi^2} U \wedge V\\
   \mathcal{R}(Y \wedge U) &=& -\frac{\phi' \psi'}{\phi \psi} Y \wedge U
  \end{eqnarray*}

This shows that, at a point $(r, p, q)$, there exists a basis $\{E_i\}$ of the tangent space such that the following hold:
	\begin{itemize}
	\item The $E_i$ are eigenvectors of $\Hess f$, and
	\item The $E_i \wedge E_j$ for $i<j$ are eigenvectors of $\mathcal R$.
	\end{itemize}
In this setting, we will use the following algebraic lemma to show that certain doubly warped products on the sphere have positive weighted sectional curvature. The proof is algebraic and is postponed until the next subsection.

\begin{corollary}\label{cor:Minimizingsecf}
Let $(M,g)$ be a closed Riemannian manifold with non-negative curvature operator $\mathcal R$. Let $X$ be a vector field on $M$. Assume that, for all $p\in M$, the tangent space at $p$ has a basis $\{E_i\}$ such that all of the following hold:
	\begin{itemize}
	\item $E_i$ is an eigenvector for $L_Xg$ with eigenvalue $\mu_i$ for all $i$,
	\item $E_i \wedge E_j$ is an eigenvalue for $\mathcal R$ with eigenvalue $\la_{ij}$ for all $i < j$, and
	\item $\la_{ij} > 0$ or $\min(\mu_i, \mu_j) > 0$ for all $i < j$.
	\end{itemize}
There exists $\la > 0$ such that $(M,g,\la X)$ has positive weighted sectional curvature.
\end{corollary}

More geometrically, this result allows us to conclude that $\sec_{\la f} > 0$ by testing this condition on orthonormal pairs of the form $(E_i, E_j)$ or $(E_j, E_i)$ with $i < j$.
\begin{proposition} 
For any positive integers $m$ and $k$, there is a doubly warped product metric on $\s^{k+m+1}$ of the form $ g = dr^2 + \phi^2(r) g_{\s^k} + \psi^2(r)g_{\s^m}$ with $\sec_f >0$ but which does not have $\sec \geq 0$. 
\end{proposition}

\begin{proof}
Let $r$ vary over the interval $[0, \pi/2]$, choose $\phi$ and $f$ as in the proof of Proposition \ref{pro:RotSym}, and set $\psi(r) = \cos(r)$. The proof of Proposition \ref{pro:RotSym} shows that we can scale $f$ so that the weighted sectional curvatures of the pairs involving $\partial_r$ and $Y$ are positive. For this argument, we apply Corollary \ref{cor:Minimizingsecf}. 

Choose an orthonormal basis $\{E_i\}_{i=0}^{k+m}$ for the tangent space such with $E_0 = \partial_r$, with $E_1,\ldots,E_{k}$ tangent to $\s^k$, and with $E_{k+1},\ldots, E_{k+m}$ tangent to $\s^m$. This basis satisfies the first two conditions of Corollary \ref{cor:Minimizingsecf}. It suffices to check the third condition.

Using the expressions above for the curvature operator, all $\la_{ij} > 0$, except in the case where $r \in [0,\pi/6]$ and where $E_i$ and $E_j$ are tangent to the $\s^k$ factor. For these indices, however, $\mu_i = \Hess f(E_i, E_i) > 0$ and $\mu_j = \Hess f(E_j,E_j) > 0$. By Corollary \ref{cor:Minimizingsecf} we have $\sec_{\la f} >0$ for some $\la > 0$. The fact that we can make $\sec<0$ for some two-planes follows for the same reason it was true in the rotationally symmetric case. 
\end{proof}

Applying O'Neill's formula from Section \ref{sec:Submersions}, this also gives us an example on $\C\pp^n$. 

\begin{proposition}\label{pro:CPnExample} There are cohomogeneity one metrics on $\C\pp^n$ which admit a density such that $\sec_f >0$ but which do not have $\sec \geq 0$. \end{proposition}

\begin{proof}
Consider a double warped product metric on the sphere $\s^{2n+1}$ of the form
\[g = dr^2 + \phi^2(r) g_{\s^{2n-1}} + \psi^2(r) d\theta^2 \]
Consider the Hopf fibration on $\s^{2n-1}$ and write the metric $g_{\s^{2n-1}} = k+h$ where $h$ is the metric tangent to the Hopf fiber and $k$ is the metric on the orthogonal complement. Complex multiplication on the $\s^{2n-1}$ and $\s^1$ factors induces a free isometric action on $g$ and the quotient is $\C\pp^n$. The quotient map is a Riemannian submersion if we equip $\C\pp^n$ with the metric 
\[ dr^2 + \phi^2(r)k + \frac{(\phi(r) \psi(r))^2}{\phi^2(r) + \psi^2(r) } h \]
By O'Neill's formula (Theorem \ref{thm:submersions}), we know this metric also has $\sec_f >0$.  Note also that if $Y$ is a horizontal vector field in the $\s^{2n-1}$ factor then for $r>0$, $[\partial_r, Y] = 0$ which implies that the sectional curvature $\sec_f^{\partial_r}(Y)$ does not change under the submersion. Since there are curvatures in the doubly warped product of this form which are negative, we also obtain that the metric on $\C\pp^n$ has some negative sectional curvatures. 
\end{proof}

\subsection{Proof of Corollary \ref{cor:Minimizingsecf}}\label{sec:Computations}

This section is devoted to the proof of Corollary \ref{cor:Minimizingsecf}, which is applied in the previous section. The proof is algebraic and not required for the rest of the paper, so the reader may choose to skip this subsection. The result is restated here for convenience:

\begin{nonumbercorollary}[Corollary \ref{cor:Minimizingsecf}]
Let $(M,g)$ be a closed Riemannian manifold with non-negative curvature operator $\mathcal R$. Let $X$ be a vector field on $M$. Assume that, for all $p\in M$, the tangent space at $p$ has a basis $\{E_i\}$ such that all of the following hold:
	\begin{itemize}
	\item $E_i$ is an eigenvector for $L_Xg$ with eigenvalue $\mu_i$ for all $i$,
	\item $E_i \wedge E_j$ is an eigenvalue for $\mathcal R$ with eigenvalue $\la_{ij}$ for all $i < j$, and
	\item $\la_{ij} > 0$ or $\min(\mu_i, \mu_j) > 0$ for all $i < j$.
	\end{itemize}
There exists $\la > 0$ such that $(M,g,\la X)$ has positive weighted sectional curvature.

\end{nonumbercorollary}

To prove this result, first note that it suffices to prove that a $\la > 0$ as in the conclusion exists at every point in $M$. It is then straightforward to conclude this pointwise claim from the following lemma together with the non-negativity of the curvature operator. 

\begin{lemma}\label{lem:MinimizingsecfHACK}
Let $(V, \inner{\cdot,\cdot})$ be a finite-dimensional inner product space. Let $\mathcal L$ and $\mathcal R$ be symmetric, linear maps on $V$ and $\Lambda^2 V$, respectively. Assume there exists an orthonormal eigenbasis $\{E_i\}$ for $\mathcal L$ such that $\{E_i \wedge E_j\}_{i<j}$ is an eigenbasis for $\mathcal R$. Denote the corresponding eigenvalues by $\mu_i$ and $\la_{ij}$, respectively. Set $\la_{ji} = \la_{ij}$ for $i<j$. Considered as a function of orthonormal pairs $(Y,Z)$ in $V$, the minimum and maximum values of
	\[\mathcal S(Y,Z) = \inner{\mathcal R(Y \wedge Z), Y \wedge Z} + \inner{\mathcal L(Y),Y}\]
lie in the set
	\[\left\{\la_{ij}+\mu_i \st i,j~\mathrm{distinct}\right\}
	\cup \left\{\frac{1}{2}\of{\la_{ij}+\la_{kl}+\mu_i+\mu_j} \st i,j,k,l~\mathrm{distinct}\right\}.\]
\end{lemma}

\begin{proof}
Let $n = \dim(V)$. Let $Y = \sum a_i E_i$ and $Z = \sum b_i E_i$ be an orthonormal pair in $V$. Observe that
	\[\mathcal S(Y,Z)
	= \sum_{i < j} \la_{ij} z_{ij} + \sum_i \mu_i x_i
	= S(x_i, z_{ij}),\]
where $x_i = a_i^2$ for $1 \leq i \leq n$, where $z_{ij} = (a_ib_j - a_jb_i)^2$ for $1 \leq i < j \leq n$. To simplify notation later, set $z_{ii} = 0$ and $z_{ji} = z_{ij}$ for $1 \leq i < j \leq n$. By orthonormality of $(Y,Z)$, all of the following hold:
	\begin{enumerate}
	\item $x_i \geq 0$ and $\sum x_i = 1$, hence the vector $x = (x_i)$ lies on the standard simplex $\Delta^{n-1} \subseteq \R^n$.
	\item Likewise, $z = (z_{ij})$ lies on the standard simplex $\Delta^{\binom{n}{2}-1} \subseteq \R^{n(n-1)/2}$.
	\item For all $1 \leq i \leq n$, $x_i \leq \sum_{j=1}^n z_{ij}$.	
	\end{enumerate}
Hence $\mathcal S(Y,Z)$ equals $S(x,z)$ for some point $(x,z)$ in the convex polytope $C \subseteq \R^n \times \R^{n(n-1)/2}$ defined by
	\[C = \left\{(x,z) \in \Delta^{n-1} \times \Delta^{\binom{n}{2}-1} \st x_i \leq \sum_j z_{ij} \mathrm{~for~all~} i\right\}.\]
To prove the lemma, it suffices to show that the function $S:C \to \R$ has extremal values in the set described in the conclusion of the lemma.

We prove this claim by induction over $n$. First, if $n = 2$, then $C = \Delta^1 \times \Delta^0$, so
	\[S(x,z) = \la_{12} + \mu_1 x_1 + \mu_2 x_2\]
has extremal values $\la_{12} + \mu_1$ and $\la_{12} + \mu_2$, as claimed. Assume now that $n \geq 3$ and that the claim holds in dimension $n-1$.

Since $C$ is a convex polytope -- i.e., an intersection of half-spaces -- and since $S$ is linear, the extremal values are attained at the corners (or $0$--dimensional faces) of $C$. We now evaluate $S$ at these corners.

Let $(x,z) \in C$ be a corner. There exist $0 \leq k \leq n$ and distinct indices $i_1,\ldots,i_k$ such that all of the following hold:
	\begin{enumerate}
	\item $(x,z)$ lies in the interior of a $k$--dimensional face of $\Delta^{n-1} \times \Delta^{\binom{n}{2}-1}$,
	\item $x_{i_h} = \sum_{j=1}^{n} z_{i_h j}$ for $1 \leq h \leq k$, and
	\item $x_i \leq \sum_{j=1}^n z_{ij}$ for all $1 \leq i \leq n$.
	\end{enumerate}
Indeed, each corner of $C$ is obtained by intersecting some $k$--dimensional face of $\Delta^{n-1} \times \Delta^{\binom n 2}$ with some choice of $k$ hyperplanes $x_i = \sum_j z_{ij}$. Recall that a $k$--dimensional face of the product is a product of a $l$--dimensional face with a $(k-l)$--dimensional face for some $0 \leq l \leq k$. Also recall that a $k$--dimensional face of a standard simplex is given by a choice of $k+1$ indices $i_0,\ldots,i_k$ for which $x_{i_0} + \ldots + x_{i_k} = 1$ and all other $x_i = 0$. Moreover, the interior of this face is the set of such points where, in addition, each of the $x_{i_h} > 0$.

First, suppose that $k = 0$. In other words, suppose that $(x,z)$ lies on a corner of $\Delta^{n-1} \times \Delta^{\binom{n}{2}}$. There exists $i$ and $p<q$ such that $x_i = 1$, $z_{pq} = 1$, and all other entries of $x$ and $z$ are zero. By condition (3), $i \in \{p,q\}$, hence $S(x,z)$ equals $\la_{iq} + \mu_i$ or $\la_{pi} + \mu_i$, as required.

Second, suppose that $k \geq 1$ and that there exists $i_h$ with $x_{i_h} = 0$. By conditions (1) and (2), $z_{i_h j} = 0$ for all $j$. Hence $S(x,z)$ does not contain any terms with index $i_h$. The claim follows in this case by the induction hypothesis.

Finally, suppose that $k \geq 1$ and $x_{i_h} > 0$ for all $1 \leq h \leq k$. In particular, $x$ does not lie in a face of dimension less than $k-1$. Hence Condition (1) implies that $x$ lies in the interior of a $(k-1)$-- or $k$--dimensional face of $\Delta^{n-1}$, and that $z \in \Delta^{\binom n 2 - 1}$ lies in the interior of a $1$--dimensional face or a corner, respectively. We consider these cases separately:
	\begin{enumerate}
	\item[(a)] In the first case, there exists $i_0 \not\in\{i_1,\ldots,i_k\}$ such that $x_{i_0} > 0$ and $x_{i_0} + x_{i_1} + \ldots + x_{i_k} = 1$. Moreover, there exists $p<q$ such that $z_{pq} = 1$ and $z_{rs} = 0$ for all $(r,s) \neq (p,q)$. By condition (3), $i_0 \in \{p,q\}$ and likewise for all of the distinct indices $i_0,i_1,\ldots,i_k$. It follows that $k$ cannot be larger than one. Moreover, if $k = 1$, then $\{i_0,i_1\} = \{p,q\}$, so 
		\[S(x,z) = \la_{i_0 i_1} + \mu_{i_0} x_{i_0} + \mu_{i_1} x_{i_1}.\]
Since $x_{i_0}$ and $x_{i_1}$ are positive and sum to one, this quantity is at least $\la_{i_0 i_1} + \mu_{i_0}$ or $\la_{i_0 i_1} + \mu_{i_1}$, as required.
	\item[(b)] In the second case, $x_{i_1} + \ldots + x_{i_k} = 1$ and there exists $p<q$ and $r<s$ such that $z_{pq} > 0$, $z_{rs} > 0$, and $z_{pq} + z_{rs} = 1$. By condition (3), $i_h \in \{p,q\} \cup \{r,s\}$ for all $h$, so clearly $k\leq 4$. In fact, if $k \geq 3$, then there exist $i_{h_1} \in \{p,q\}$ and $i_{h_2} \in \{r,s\}$, which implies
		\[1 = \sum_{h=1}^k x_{i_h} > x_{i_{h_1}} + x_{i_{h_2}} = \sum_j z_{i_{h_1} j} + \sum_j z_{i_{h_2} j} \geq z_{pq} + z_{rs} = 1,\]
a contradiction. 

This leaves the possibilities that $k=2$ and $k = 1$. First, suppose $k = 1$. It follows that $x_{i_1} = 1$ and that 
	\[S(x,z)=\la_{pq} z_{pq} + \la_{rs} z_{rs} + \mu_{i_1}.\] 
Hence $S(x,z)$ is bounded between $\la_{pq}+\mu_{i_1}$ and $\la_{rs}+\mu_{i_1}$. Moreover, 
	\[1 = x_{i_1} = \sum_j z_{i_1 j},\]
so all $z_{ij}$ that do not appear in this sum are zero. In particular, $i_1 \in \{p,q\}$ and $i_1 \in \{r,s\}$, so the claim follows in this case.

This leaves the case with $k=2$. We start by showing that $i_1$ cannot be in both $\{p,q\}$ and $\{r,s\}$. Indeed, if it were, then Conditions (1) and (2) imply that
	\[1 = x_{i_1} + x_{i_2} > x_{i_1} = \sum_j z_{i_1 j} \geq z_{pq} + z_{rs} = 1,\]
a contradiction. By a similar argument, $i_2$ cannot be in both sets. Condition (3) implies $i_1,i_2 \in \{p,q\} \cup \{r,s\}$. If $i_1$ and $i_2$ lie in different sets, say $i_1 \in \{p,q\}$ and $i_2 \in \{r,s\}$, then Condition (2) further implies that $x_{i_1} = z_{pq}$ and $x_{i_2} = z_{rs}$, hence
	\[S(x,z) = \of{\la_{pq} + \mu_{i_1}} z_{pq}
						 + \of{\la_{rs} + \mu_{i_2}} z_{rs},\]
so the claim follows in this case. Finally, if $i_1$ and $i_2$ lie in the same set, say $\{p,q\}$, then
	\[S(x,z) = \la_{i_1 i_2} z_{i_1 i_2} + \la_{rs} z_{rs}
						 + \mu_{i_1} x_{i_1} + \mu_{i_2} x_{i_2}.\]
Moreover, in this case, condition (2) implies $x_{i_1} = z_{pq} = x_{i_2}$, and condition (1) implies that $1 = x_{i_1} + x_{i_2} = 2z_{pq}$, hence all four variables are equal to $1/2$. This concludes the proof of the claim.
	\end{enumerate}

This shows in all cases that the extremal vaues of $S : C \to \R$ are given as in the conclusion of the lemma. As established at the beginning of the proof, the same holds of $\mathcal S$.
\end{proof}

Regarding the proof of Lemma \ref{lem:MinimizingsecfHACK}, we note that the point $(x,z)$ with $x_1 = x_2 = z_{12} = z_{34} = \frac 1 2$ and all other entries zero lies in the set $C$. Moreover, since the $\la_{ij}$ and $\mu_i$ are arbitrary, we have provided the optimal solution the the optimization problem for the function $S : C \to \R$. On the other hand, $\mathcal S(Y,Z)$ actually equals $S(x,z)$ for some $(x,z) \in C_0$, where $C_0$ is a proper subset of $C$. Indeed, given the definitions of $x_i$ and $z_{ij}$ as in the proof, it is straightforward to check that
	\begin{enumerate}
	\item[(4)] $z_{ij} \leq x_i + x_j$ for all $i < j$.
	\end{enumerate}
Note that the point $(x,z)$ with $x_1 = x_2 = z_{12} = z_{34} = \frac 1 2$ is not in the smaller set $C_0$. This suggests that the Lemma \ref{lem:MinimizingsecfHACK} could be improved to state that the optimal values are of the form $\la_{ij} + \mu_i$ or $\la_{ij} + \mu_j$ with $i < j$. Since this is not needed for our applications, we do not pursue this here.

\bigskip
\section{Averaging the density}\label{sec:Averaging}
\bigskip

In this section, we begin to establish the properties of positive weighted sectional curvature described in Section \ref{sec:Preliminaries}. Our first consideration is that, in studying manifolds  with density and symmetry, a symmetry of the metric might not be a symmetry of the density. We prove in this section that this difficulty can be overcome in the compact case. At the end, we apply these ideas to study weighted curvature properties of homogeneous metrics.

\subsection{Preservation of weighted curvature bounds under averaging}

Fix a Riemannian manifold $(M,g)$ and a vector field $X$ on $M$. Let $G$ be a compact subgroup of the isometry group, and let $d\mu$ denote a unit-volume, bi-invariant measure on $G$. Define a new, $G$--invariant vector field $\bar{X}$ on $M$ as follows:
	\[\bar{X}_p = \int_G \phi_*^{-1}(X_{\phi(p)}) d\mu,\]
where we identify the elements $\phi \in G$ with isometries $\phi\colon M\to M$. In the gradient case, where $X = \nabla f$, we similarly define $\bar f(p) = \int_G f(\phi(p)) d\mu$. 

As a basic observation note that, for a fixed vector field $V$ in $T_pM$, 
\begin{eqnarray*}
g\of{\bar{X}, V } &=& g\of{ \int_G \phi_*^{-1}(X)d\mu, V} = \int_G g\of{\phi_*^{-1}(X), V}d\mu, \\
D_V \of{\int_G g\of{\phi_*^{-1}(X), V}d\mu} &=& \int_G g\of{\nabla_V \phi_*^{-1}(X), V} d\mu + \int_G g\of{\phi_*^{-1}(X), \nabla_V V} d\mu.
\end{eqnarray*}
This follows from the fact that all of the functions involved are smooth, the linearity of the integral, and the fact that $G$ as a compact space admits a finite partition of unity. Similar identities for passing integrals over $G$ past a derivative also hold for the same reasons. We will use these facts repeatedly below with out further comment. 

Now we claim the following:

\begin{lemma} \label{lem:LieDerivative}
With the notation above, for any vector field $X$ and any $V\in T_pM$, 
\[ (L_{\bar{X}} g)(V,V) = \int_G (L_X g) ( \phi_*V, \phi_* V) d\mu \]
\end{lemma}

\begin{proof}
This follows from a straightforward calculation:
	\begin{eqnarray*}
	g\of{\nabla_V \bar X, V}
	&=& D_Vg\of{\bar X, V} - g\of{\bar X, \nabla_V V}\\
	&=& D_Vg\of{\int_G \phi_*^{-1}(X)d\mu, V} - g\of{\int_G \phi_*^{-1}(X)d\mu, \nabla_V V}\\
	&=& \int_G D_Vg\of{\phi_*^{-1}(X), V}d\mu - \int_G g\of{\phi_*^{-1}(X), \nabla_V V}d\mu\\
	&=& \int_G g\of{\nabla_V\of{\phi_*^{-1}(X)}, V} d\mu\\
	&=& \int_G g\of{\nabla_{\phi_* V} X, \phi_*V} d\mu.
	\end{eqnarray*}
\end{proof} 

For a function we also have the following. 

\begin{lemma} \label{lem:Hessian}
With the notation above, for any function $f$, 
\begin{eqnarray*}
\nabla\bar f &=& \overline{\nabla f}\\
\Hess \bar f &=& \int_G \Hess f ( \phi_*V, \phi_* V) d\mu 
 \end{eqnarray*}
\end{lemma}
\begin{proof} First note that the second equation follows from the first combined with Lemma \ref{lem:LieDerivative} along with the fact that 
\[ \Hess f = \frac 1 2 L_{\nabla f} g .\] 

To prove the first equation, let $V$ be a vector field on $M$, and observe that
	\[g\of{\nabla \bar f, V}
	= D_V\of{\int_G f\circ \phi d\mu}
	= \int_G df(\phi_* V) d\mu
	= \int_G g\of{\phi_*^{-1}(\nabla f), V}
	= g\of{\overline{\nabla f}, V}.\]
\end{proof}

Now we are ready to show that the weighted curvatures can be averaged over the compact group $G$.  First we consider the $\infty$--cases.

\begin{lemma}\label{lem:averaging}
Given a triple $(M,g,X)$ and a compact subgroup $G$ of the isometry group, the weighted curvatures satisfy
\begin{eqnarray*}
\Ric_{\bar X}(U,V) &=& \int_G \Ric_X(\phi_*U, \phi_*V) d\mu, \\
\sec_{\bar X}^V(U) &=& \int_G \sec_X^{\phi_*V}(\phi_*U) d\mu,
\end{eqnarray*}
where $\bar X$ is the average of $X$. In particular, if $\sec_X \geq \lambda$, then $\sec_{\bar X} \geq \la$ where $\bar X$ is $G$--invariant. 
\end{lemma}

\begin{remark} One similarly can draw conclusions about upper bounds and for the Bakry--Emery Ricci curvature. In addition, analogous statements hold for $\Ric_f$ and $\sec_f$. They follow immediately from Lemmas \ref{lem:Hessian} and \ref{lem:averaging} .\end{remark}

\begin{proof} 
Using Lemma \ref{lem:LieDerivative} we can see all we need to show is 
\begin{eqnarray*} 
\Ric(U,V) &=& \int_G \Ric(\phi_*U, \phi_*V) d\mu\\
\sec(U,V) &=& \int_G \sec(\phi_*U, \phi_*V) d\mu.
\end{eqnarray*} 
But this just follows from the isometry invariance of the curvature as well as the fact that $d\mu$ has unit volume. 
\end{proof}

For the strongly weighted curvatures, averaging the vector field $X$ causes some issues as the equation contains terms which are quadratic in $X$. In the gradient case we can overcome this by changing the form of the potential function. Given $m$, set $u = e^{-f/m}$, then a simple calculation shows that 
\[ \Hess f - \frac{df \otimes df}{m} = -\frac{ m \Hess u }{u} \]
So, we have
\begin{eqnarray*}
 \Ric_f^m = \Ric -\frac{ m \Hess u }{u}
\end{eqnarray*}
and, choosing $m = -1$,
\begin{eqnarray*}
\overline{\sec}^V_f(U) = \sec(V,U)+ \frac{\Hess u}{u}(V,V) 
\end{eqnarray*} 
In these cases, it is natural to average the function $u$. Let $\widetilde{u}(p) = \int_G u(\phi(p)) d\mu$ and define $\widetilde{f} = -m \log(\widetilde{u})$. Then we have the following Lemma. 
 
\begin{lemma} \label{lem:u-averaging} 
Given a triple $(M,g,f)$ and a compact subgroup $G$ of the isometry group, the weighted curvatures satisfy 
	\begin{eqnarray*}
	\widetilde{u} \Ric^m_{\widetilde f} (U,V) 
		&=& \int_G u \Ric^m_f(\phi_*U, \phi_*V) d\mu, \\
	\widetilde{u} \overline{\sec}_{\widetilde f}^V(U) 
		&=& \int_G u \overline{\sec}_f^{\phi_*V}(\phi_*U) d\mu,
	\end{eqnarray*}
where $\widetilde u$ is the average of $u = e^{-f/m}$ and $\widetilde f = -m \log(\widetilde u)$. In particular, if $\overline\sec_f \geq \la$, then $\overline\sec_{\widetilde f} \geq \la$ where $\widetilde f$ is $G$--invariant.
\end{lemma}

\begin{proof} 
We will discuss the Ricci case and the sectional curvature case will follow from an analogous argument. We have 
\begin{eqnarray*}
u \Ric_f^m(V,V) &=& u \Ric(V,V) - m \Hess u(V,V) \\
\int_G u \Ric_f^m(\phi_*V,\phi_*V) d\mu &=& \int_G \of{ u \Ric(\phi_*V,\phi_*V) - m \Hess u(\phi_*V,\phi_*V) } d\mu \\
&=& \widetilde{u} \Ric(V,V) - m \Hess \widetilde{u} (V,V) \\
&=& \widetilde{u} \Ric^m_{\widetilde f} (V,V)
\end{eqnarray*}
To see the final remark note that, if $\Ric_f^m \geq \lambda g$, then 
\begin{eqnarray*} 
\Ric^m_{\widetilde f} (U,V) &=& \frac { \int_G u \Ric^m_f(\phi_*U, \phi_*V) d\mu}{\widetilde u} \geq \frac{ \int_G \lambda u g(U,V) d\mu}{\widetilde{u}} = \lambda g(U,V),
\end{eqnarray*}
so $\Ric^m_{\widetilde f} \geq \la g$ as well. Similar arguments hold for upper bounds.
\end{proof}

We remind the reader that Lemmas \ref{lem:averaging} and \ref{lem:u-averaging} immediately imply Corollary \ref{cor:PWSCaveraging}: If $(M,g,X)$ has positive weighted sectional curvature and $G$ is a compact subgroup of the isometry group of $(M,g)$, then there exists a $G$--invariant vector field $\tilde X$ such that $(M,g,\tilde X)$ has positive weighted sectional curvature. Indeed, if $\sec_X > 0$, then one can replace $X$ by its average $\bar X$ over the $G$--orbits. If $X = \nabla f$ and $\overline\sec_f > 0$, then one can replace $X$ by $\tilde X = \nabla \tilde f$, where $\tilde f = \log(\tilde u)$ and where $\tilde u$ is the average of $u=e^f$ over the $G$--orbits.

\begin{remark} Note that Lemma \ref{lem:u-averaging} does not clearly extend to the non-gradient case, since there is no globally defined function $u$ to average. We can still average over $X$, but only one side of the curvature bound is preserved. To see this note that the strongly weighted curvatures satisfy
	\[\overline{\sec}_{\bar X}^V(U) = \int_G \overline{\sec}_X^{\phi_*V}(\phi_*U) d\mu + \of{\int_G g\of{X, \phi_* V} d \mu}^2 - \int_G g\of{X, \phi_*V}^2 d \mu.\]
In particular, by the Cauchy--Schwarz inequality,
	\[\overline\sec_{\bar X}^V(U) \leq
	 \int_G \overline\sec_X^{\phi_*V}(\phi_*U) d\mu,\]
so upper bounds on strongly weighted curvatures are preserved by averaging the density. Similar statements hold in the gradient case. For the $m$--Bakry--Emery curvature, we similarly have
	\[\Ric^m_{\bar X}(V,V) = \int_G \Ric_X^m(\phi_*V, \phi_*V)d\mu - \frac{1}{m}\ofsq{\of{\int_Gg\of{X, \phi_*V}}^2 - \int_Gg\of{X,\phi_*V}^2}.\]
\end{remark}

\subsection{Homogeneous metrics}

Now we apply averaging the density to the special case of homogeneous metrics. Homogeneous Riemannian manifolds with positive sectional curvature are classified Wallach \cite{Wallach72} and B\'erard-Bergery \cite{Berard-Bergery76}. By averaging the density, we show here that there are no additional examples in the weighted case when $X = \nabla f$.

\begin{proposition} \label{pro:CompactHomogeneous}
Let $(M,g)$ be a compact homogeneous manifold and let $f\in C^\infty(M)$.
\begin{enumerate}
\item If $\Ric_f \geq \lambda g$ or $\Ric_f^m \geq \lambda g$, then $\Ric \geq \lambda g$.
\item If $\sec_f \geq \lambda g$ or $\overline{\sec}_f \geq \lambda g$ then $\sec \geq \lambda g$.
\end{enumerate}
Analogous results hold for upper bounds. 
\end{proposition}

This proposition immediately implies Theorem \ref{thm:PWSChomogeneous} from the introduction. Indeed, if $(M,g)$ admits a gradient field $X = \nabla f$ with positive weighted sectional curvature, then $\overline\sec_f > 0$ and hence this proposition applies.

\begin{proof}
Let $G$ be the isometry group of $(M,g)$. In all cases, we can replace $f$ by a $G$--invariant function $\tilde f$ such that the $\tilde f$--weighted curvatures have the same lower bounds as the $f$--weighted curvatures. Since $G$ acts transitively, $\tilde f$ is constant, so the $\tilde f$--weighted curvatures are equal to the usual, unweighted curvatures.
\end{proof}

It is not clear whether this fact is also true when the field $X$ is not gradient. Averaging the field so that it is invariant under the isometries will not necessarily make the field Killing, but there is one important case where it does. 
 
 \begin{proposition}
 If a compact Lie group with a bi-invariant metric admits an $X$ such $\sec_X \geq \la$ or $\Ric_X \geq \la g$, then $\sec\geq \la$ or $\Ric\geq \la g$, respectively.
 \end{proposition}
 
 \begin{proof}
We can replace $X$ by its average over the left and right actions of $G$. This preserves the lower bounds on curvature, and it makes $X$ bi-invariant and hence a Killing field. Hence $L_X g = 0$, so the weighted curvatures equal the unweighted curvatures.
 \end{proof}

In particular, the previous two propositions have the following corollary. 

\begin{corollary} A compact Lie group with a bi-invariant metric has positive weighted sectional curvature if and only if it has positive sectional curvature. \end{corollary}
 
In the simplest non-trivial case of a left-invariant metric that is not bi-invariant, a computation shows that we again do not get new examples. 
 
 \begin{proposition} If a left invariant metric on the Lie group $\SU(2)$ supports a vector field $X$ such that $\sec_X \geq \la$ or $\Ric_X \geq \la g$, then $\sec\geq \la$ or $\Ric\geq \la g$, respectively.
 \end{proposition}
 
\begin{proof} 
For a left invariant metric on $\SU(2)$, choose an orthonormal frame 
 \[ \lambda_1^{-1} X_1, \lambda_2^{-1} X_2, \lambda_3^{-1} X_3 \]
such that $[X_i, X_{i+1}] = 2 X_{i+2}$ with indices taken mod 3. It follows that 
 \begin{eqnarray*}
 \nabla_{X_i} X_i 
 	&=& 0 \\
 \nabla_{X_i} X_{i+1} 
 	&=& \left(\frac{ \lambda_{i+2}^2 + \lambda_{i+1}^2 - \lambda_i ^2}{\lambda_{i+2}^2} \right) X_{i+2}\\
 \nabla_{X_{i+1}} X_i 
 	&=& \left(\frac{- \lambda_{i+2}^2 + \lambda_{i+1}^2 - \lambda_i ^2}{\lambda_{i+2}^2} \right) X_{i+2}
 \end{eqnarray*}
 
Now since $\SU(2)$ is compact, we can assume by averaging that $X$ is a left-invariant vector field, which we will write as 
 \[ X = a_1 X_1 + a_2 X_2 + a_3 X_3 \]
 for constants $a_i$. We have
 \begin{eqnarray*}
 (L_Xg) (X_i, X_i) &=& 2 g(\nabla_{X_i} X, X_i) = 0\\
 (L_Xg) (X_i, X_{i+1}) &=& g(\nabla_{X_i} X, X_{i+1}) + g(\nabla_{X_{i+1}} X, X_{i}) \\
 &=& 2 a_{i+2} \left( \lambda_i^2 - \lambda_{i+1}^2 \right)
 \end{eqnarray*}
This shows that $X$ is not a Killing field in general. However, $\sec_X(X_i, X_j) = \sec(X_i, X_j)$, so if $\sec_X \geq \lambda$ then $\sec(X_i, X_j) \geq \lambda$. Further computation also shows that the basis $X_1 \wedge X_2$, $X_2 \wedge X_3$, $X_3 \wedge X_1$ diagonalizes the curvature operator, and thus that all of the sectional curvatures are bounded by the maximum and minimum curvatures of the sectional curvatures involving $X_1, X_2,$ and $X_3$. Thus we actually have $\sec \geq \lambda$.  The basis $X_1, X_2, X_3$ also diagonalizes the Ricci tensor so the statement about Ricci curvatures follows similarly. 
 \end{proof}
 
 In general, Proposition \ref{pro:CompactHomogeneous} does not hold in the non-compact case, as we have already seen in Example \ref{Ex:Gaussian}. We can generalize the Gaussian example in the following simple way: 
    \begin{example} \label{Ex:GaussianGeneralization} 
 Suppose that $(M,g)$ is a simply connected space of non-positive sectional curvature. The distance function to a point squared, $d^2$, is a smooth function. Moreover, 
 $\Hess (d^2) \geq 2 g$. Therefore, if $(M,g)$ has sectional curvature bounded from below by $-K$, then, for $f = A d^2$, we have $\sec_f \geq 2A-K$, which we can make arbitrarily large. 
  \end{example}
 Letting $(M,g)$ in the example be a hyperbolic space gives a  noncompact homogeneous manifold with positive weighted sectional curvature and negative sectional curvature. 
 We also note that there are many examples of non-compact homogeneous Ricci soliton metrics (i.e metrics with $\Ric_X = \lambda g$) which do not have $\Ric = \lambda g$. Examples of homogeneous metrics with $\Ric_f^m = \lambda g$ which do not have $\Ric = \lambda g$ are also constructed in \cite{HePetersenWylie-pre}.

\bigskip
\section{Riemannian submersions and Cheeger deformations}\label{sec:Submersions}
\bigskip

We analyze the behavior of the weighted and strongly weighted directional curvature operators under a Riemannian submersion $\pi:M \to B$. For this, we restrict to vector fields $X$ on $M$ for which the vector field $\pi_*(X)$ on $B$ is well defined. Following Besse \cite[Chapter 9]{Besse-Einstein}, let $R$, $\hat{R}$, and $\check{R}$ denote the curvature tensors of $M$, the fibers, and the base, respectively, and let $\mathcal V$ and $\mathcal H$ denote the projection maps onto the vertical and horizontal spaces, respectively.

\begin{theorem}[O'Neill formulas]\label{thm:submersions}
Let $(M,g)$ be a closed Riemannian manifold, let $\pi$ be a Riemannian submersion with domain $M$, and let $X$ be a smooth vector field on $M$ such that the map $p \mapsto \pi_*(X_p)$ is constant along the fibers of $\pi$. If $Y$ and $Z$ are horizontal vector fields and $U$ and $V$ are vertical vector fields on $M$, then
	\begin{eqnarray*}
	R_X^V(U,U) &=& \hat R_{\mathcal V X}^V(U,U) + g\of{T_U V,T_U V} - g\of{T_U U, T_V V} - g\of{T_V V, \mathcal H X}g(U,U),\\
	R_X^Z(Y,Y) &=& \check R_{\pi_* X}^{\pi_*Z}(\pi_*Y, \pi_*Y) - 3g(A_Y Z, A_Y Z),
	\end{eqnarray*}
and likewise with $R_X$, $\hat R_{\mathcal V X}$ and $\check R_{\pi_* X}$ replaced by the strongly weighted directional curvature operators on $M$, the fibers, and the base, respectively.

In particular, if $(Y, Z)$ is an orthonormal pair of horizontal vector fields, then
	\[\sec_{\pi_*X}^{\pi_* Y}(\pi_* Z)
		= \sec_X^Y(Z) + \frac{3}{4}\left|[Y,Z]^{\mathcal V}\right|^2\]
and likewise for $\overline\sec_X$.
\end{theorem}

We remark that analogous statements hold in the gradient case. There, one assumes that $f$ is a smooth function on $M$ that is constant along the fibers of $\pi$. The function $f$ replaces $X$ in the above formulas, and the induced map $\bar f$ on the base replaces $\pi_*X$. The gradient case follows from the general case since $d\bar f$ and $\Hess \bar f$ pull back via $\pi$ to $df$ and $\Hess f$, respectively.

We also remark that, as with sectional curvature, the base of a Riemannian submersion inherits lower bounds on weighted or strongly weighted sectional curvatures. In particular, if the total space admits a vector field $X$ with positive weighted sectional curvature such that $X$ descends to a well defined vector field on the base, then the base too has positive weighted sectional curvature (see Corollary \ref{cor:PWSCSubmersion}).

Finally, we remark that the vector field $X$ is arbitrary and hence need not be horizontal or vertical. For example, suppose $\pi$ is the quotient map by a free, isometric group action. The vector field $X$ might be an action field (hence vertical), basic (hence horizontal), or any smooth combination of the two (hence neither).

\begin{proof}
Let $\hat g$ and $\check g$ denote the metrics on the fibers and the base, respectively. First, the conclusions in the strongly weighted cases follows immediately from the weighted cases since
  \begin{eqnarray*}
  g\of{X, V}^2 g\of{U,U} 
  &=& \hat g(\mathcal V X, V)^2 \hat g(U,U),\\
  g\of{X, Z}^2 g\of{Y,Y}
  &=& \check g(\pi_* X, \pi_*Z)^2 \check g(\pi_*Y,\pi_*Y).
  \end{eqnarray*}
Second, the weighted cases follow from the unweighted case once we establish that
	\begin{eqnarray*}
	\frac 1 2 (L_X g)(V,V)g(U,U) 
	&=& \frac 1 2 (L_{\mathcal V X} \hat g)(V,V) \hat g(U,U) 
	- g\of{T_V V, \mathcal H X}g\of{U,U},\\
	\frac 1 2 (L_X g)(Z,Z)g(Y,Y) 
	&=& \frac 1 2 (L_{\pi_*X} \check g)(\pi_*Z,\pi_*Z)\check g(\pi_*Y,\pi_*Y).
	\end{eqnarray*}				 
Indeed these follow from the fact that $U$ is vertical, the fact that $Y$ is horizontal, and the observations that
	\begin{eqnarray*}
	\frac 1 2 (L_X g)(V,V)
	= g\of{\nabla_V X, V}
	&=& g\of{\nabla_V(\mathcal V X), V}
	 + g\of{\nabla_V(\mathcal H X), V}\\
	&=& \hat g\of{\hat \nabla_V(\mathcal V X), V}
	 - g\of{\nabla_V V, \mathcal H X}\\
	&=& \frac 1 2 (L_{\mathcal V X} \hat g)(V,V)
	 - g\of{T_V V, \mathcal H X}
	\end{eqnarray*}
and
\[	\frac 1 2 (L_X g)(Z,Z) 
	= g\of{\nabla_Z X, Z} 
	= \check g(\nabla_{\pi_* Z} \pi_*X, \pi_*Z)\\ 
	= \frac 1 2 (L_{\pi_*X} \check g)(\pi_*Z, \pi_*Z).\]

\end{proof}

Regarding the O'Neill formulas for mixed inputs (vertical and horizontal), we remark that one simply obtains weighted versions by adding the appropriate terms from the definition of $R_X$ and $\overline R_X$. The formulas do not simplify as in Theorem \ref{thm:submersions}, but one can still use them. To illustrate this with one easy example, we generalize here a result of Weinstein \cite[Theorem 6.1]{Weinstein80} to the case of positive weighted sectional curvature (cf. Florit and Ziller \cite{FloritZiller11} and Chen \cite{Chen14-thesis}).

\begin{theorem}[Weinstein]\label{thm:WeinsteinFloritZiller}
Let $\pi:M\to B$ be Riemannian submersion of closed Riemannian manifolds with totally geodesic fibers. If there exists a function $f \in C^\infty(M)$ such that $\overline\sec_f > 0$ on all orthonormal pairs of vectors spanning ``vertizontal'' planes, 
then the fiber dimension is most $\rho(\dim B)$, where $\rho(n)$ denotes the maximum number of linearly independent vector fields on $\s^{n-1}$. \end{theorem}

Note that this reduces to the Weinstein's result when $f = 0$. Recall that $f \in C^\infty(M)$ is basic if it is constant along the fibers of $\pi$.

\begin{proof}
Since the fibers are totally geodesic, the $T$ tensor vanishes. Hence, for any orthonormal pair $(V,Z)$, where $V$ is vertical and $Z$ is horizontal, the O'Neill formula
	\[R(Z,V,V,Z) = |A_Z V|^2 - |T_Z U|^2 + g\of{(\nabla_Z T)_V V, Z}\]
implies
	\[\overline\sec_f^V(Z) = |A_Z V|^2 + \Hess f(V,V) + df(V)^2.\]
At a point $p \in M$ where $f$ is maximized, $df(V) = 0$ and $\Hess f(V,V) \leq 0$ for all $V$. Hence, $A_Z V \neq 0$ for all vertizontal pairs $(V,Z)$ at $p$. The proof now proceeds as in \cite{Weinstein80} by constructing $\dim(\mathcal V_p)$ linearly independent vectors on the unit sphere in $\mathcal H_p$, where $\mathcal V_p$ and $\mathcal H_p$ are the vertical and horizontal spaces at $p$, respectively.
\end{proof}

Theorem \ref{thm:WeinsteinFloritZiller} relates to a conjecture of Fred Wilhelm, namely, that $\dim(F) < \dim(B)$ for any Riemannian submersion from a manifold $M$ with positive sectional curvature, where $\dim(F)$ and $\dim(B)$ denote the dimensions of the fibers and the base, respectively. If one only assumes $\sec > 0$ almost everywhere on $M$, then there are counterexamples due to Kerin \cite{Kerin11}. On the other hand, the above result suggests that the assumption of positive sectional curvature might be weakened to cover manifolds with density. For example, Frankel's theorem (Theorem \ref{thm:Frankel}) in the weighted case implies the following: if $M$ admits a vertical vector field $X$ such that $M$ has positive weighted sectional curvature, then the conclusion of Wilhelm's conjecture holds.

As a second application, we discuss Cheeger deformations. These have been used in multiple constructions of metrics with positive or non-negative sectional curvature (see Ziller \cite{Ziller07} for a survey). Here, we establish the weighted curvature formulas for the deformed metric in terms of the original. We will use the formulas from this section in the proof of Theorem \ref{thm:IntroConnectedness}.

The setup involves a Riemannian manifold $(M,g)$, a subgroup $G$ of the isometry group, a bi-invariant metric $Q$ on $G$, and a real parameter $\lambda > 0$. We are interested in understanding how the weighted curvatures behave under these peturbations. Hence we also fix a smooth vector field $X$ on $M$. We assume that $X$ is $G$--invariant, which can be arranged if the subgroup $G$ is compact, e.g., if $G$ is closed and $M$ is compact.

The new metric on $M$ is denoted by $g_\lambda$. It is the metric for which the map
	\[\pi\colon(G \times M, \lambda Q + g) \to (M, g_\lambda)\]
given by $(h,p) \mapsto h^{-1}p$ is a Riemannian submersion. 
There is a $(\lambda Q+g)$--orthogonal decomposition of $T_{(e,p)}(G\times M)$ as
	\[   \{(Y, Y_p^*) \st Y \in \mathfrak g\} 
	 \oplus \left\{\of{- |Y^*_p|^2 Y, \la|Y|^2 Y^*_p} \st Y \in \mathfrak g\right\}
	 \oplus \{(0, Z)   \st Z \in T_p(G\cdot p)^\perp\}.\]
Here, and throughout, $\mathfrak g = T_eG$ denotes the Lie algebra of $G$, and $Y^*$ denotes the Killing field associated to $Y \in \mathfrak{g}$. The first of these summands is the vertical space $\mathcal V_{(p,e)} = \ker(D\pi_{(e,p)})$ of the projection $\pi$. The last two summands together form the horizontal space $\mathcal H_{(e,p)} = \mathcal V_{(e,p)}^\perp$.

The horizontal lift of $Y^*\in T_p(G\cdot p) \subseteq T_p M$ is 
\[\frac{1}{|Y^*_p|^2 + \la |Y|^2}\of{- |Y^*_p|^2 Y, \la|Y|^2 Y^*_p},\]
and the horizontal lift of $Z \in T_p(G\cdot p)^\perp \subseteq T_p M$ is $(0,Z)$. Note that $|Z|_{g_\la} = |Z|_g$, while
	\[|Y^*|_{g_\la}^2 = \frac{\la |Y|^2 |Y^*|^2}{|Y^*|^2 + \la |Y|^2}.\]
As $\la \to \infty$, $|Y^*|_{g_\la}$ increases and converges to $|Y^*|_g$, hence $|Y^*|_{g_\la} \leq |Y^*|$. We will use this in the proof of the connectedness lemma.

Our goal now is to compute the weighted and strongly weighted directional curvature operators of $(M,g_\lambda, X)$ in terms of those of $(M, g, X)$.

\begin{lemma}[Curvature tensors after Cheeger deformations] \label{Lemma:CheegerDeformation}
Let $R = R^g$ and $R^{g_\la}$ denote the curvature tensors of $(M, g)$ and $(M, g_\la)$, respectively. For vector fields $W_i$ on $M$, if $\tilde W_i = (\tilde W_i^G, \tilde W_i^M)$ denote the horizontal lifts in $G \times M$, then
	\begin{eqnarray*}
	g_\la\of{(R^{g_\la})_X^{W_1}(W_2), W_3}
	&=& \la Q\of{(R^Q)^{\tilde W_1^G}(\tilde W_2^G), \tilde W_3^G}\\
	&~& +  g\of{(R^g)_X^{\tilde W_1^M}(\tilde W_2^M), \tilde W_3^M}\\
	&~& + (\la Q + g)\of{A_{\tilde W_1} \tilde W_2,
				  A_{\tilde W_1} \tilde W_2}.
	\end{eqnarray*}
In particular, if $Z_1$ and $Z_2$ are vector fields in $M$ that are everywhere orthogonal to the $G$--orbits, then
	\[
	 g_\la\of{(R^{g_\la})_X^{Z_1}(Z_2), Z_2}
	 \geq g\of{   (R^g)_X^{Z_1}(Z_2), Z_2}.
	\]
If, in addition, $(Z_1,Z_2)$, forms an orthonormal pair with respect to $g$ (equivalently with respect to $g_\la$), then
	\[(\sec^{g_\la})_X^{Z_1}(Z_2) \geq (\sec^g)_X^{Z_1}(Z_2).\]
\end{lemma}

\begin{proof}
Consider the vector field $(0,X)$ on $G \times M$. It is $G$--invariant and $\pi_*(0,X) = X$, where $\pi:(G\times M, \la Q + g) \to (M,g_\la)$ is the Riemannian submersion defining $g_\la$. The first claim follows directly from the (first) O'Neill formula in the weighted case applied to $\pi$. The second and third claims follow from the fact that the horizontal lift of $Z \in T_p(G \cdot p)^\perp$ is $(0,Z) \in T(G\times M)$.
\end{proof}

\bigskip
\section{Weinstein's fixed point theorem and applications}\label{sec:Weinstein}	
\bigskip

In the next two sections, we demonstrate how Synge-type arguments extend to the case of positive weighted sectional curvature. The only technical ingredient required is Lemma \ref{lem:SecondVariation}. We first prove Weinstein's fixed point theorem in the weighted case:

\begin{theorem}[Weinstein's fixed point theorem]
Let $(M^n,g)$ be a closed, orientable Riemannian manifold equipped with vector field $X$ such that $(M,g,X)$ has positive weighted sectional curvature. If $F$ is an isometry of $M$ with no fixed point, then $F$ reverses orientation if $n$ is even and preserves it if $n$ is odd.
\end{theorem}

\begin{proof}
Corollary \ref{cor:PWSCaveraging} implies that we may assume without loss of generality that $X$ is invariant under isometries. In particular, $F_*(X) = X$.

The proof now proceeds as in Weinstein \cite{Weinstein68}. Using compactness, choose $p \in M$ such that $d(p,F(p))$ is minimal. Choose a unit-speed, minimizing geodesic $\gamma:[a,b] \to M$ from $p$ to $F(p)$. As in \cite{Weinstein68}, there exists a special unit-length, parallel vector field $V$ along $\gamma$, 
and it suffices to show that the index $I(V,V)$ of $\gamma$ is negative. One of the properties of $\gamma$ is that $F_*(\gamma'(a)) = \gamma'(b)$. By Lemma \ref{lem:SecondVariation}, it suffices to show that
  \[\left.g\of{\gamma'(t), X_{c(t)}}\right|_{t=a}^{t=b}
  = \inner{\gamma'(b), X_{\gamma(b)}}
	- \inner{\gamma'(a), X_{\gamma(a)}} = 0.\]
Indeed, this is the case since $F$ carries $\gamma'(a)$ to $\gamma'(b)$ and $X_{\gamma(a)}$ to $X_{F(\gamma(a))} = X_{\gamma(b)}$.
\end{proof}

We derive three corollaries of Weinstein's theorem, all of which are analogues of what happens in the unweighted case. The first is the textbook application of Weinstein's theorem to prove Synge's theorem.

\begin{corollary}[Synge's theorem]
If $(M^n, g, X)$ is closed and has positive weighted sectional curvature, then
	\begin{itemize}
	\item If $n$ is odd, then $M$ is orientable.
	\item If $n$ is even and $M$ is orientable, then $\pi_1(M)$ is trivial.
	\end{itemize}
\end{corollary}

This is proved in \cite{Wylie-pre}, but we indicate another proof based on Weinstein's theorem. Depending on whether $n$ is odd or even, one applies Weinstein's theorem in the weighted case to the free action of $\Z_2$ or $\pi_1(M)$, respectively, on the orientation or universal cover of $M$ equipped with the pullback metric and vector field or function. For this, it is important that $\pi_1(M)$ is finite (see Theorem \ref{thm:pi1finite}).

Weinstein's theorem, together with O'Neill's formula, also provides another proof of Berger's result (see \cite{Berger66,GroveSearle94}):

\begin{corollary}[Berger's theorem]\label{cor:Berger}
If $(M^n,g,X)$ is closed and has positive weighted sectional curvature, then the following hold:
	\begin{itemize}
	\item If $n$ is even, then any Killing field has a zero. Equivalently, any isometric torus action has a fixed point.
	\item If $n$ is odd, any torus acting isometrically on $M$ has a circle orbit. In particular, there exists a codimension one subtorus that has a fixed point.
	\end{itemize}
\end{corollary}
We remark that the even-dimensional case is also proved in \cite{Wylie-pre}.
\begin{proof}
The equivalence of the conclusions about Killing fields and torus actions is based on the fact that the isometry group of $M$ is a compact Lie group. Consider an isometric action on $M$ by a torus $T$. Without loss of generality, we assume that $X$ is invariant under the action of $T$. The conclusion follows by choosing $F \in T$ that generates a dense subgroup of $T$ and applying Weinstein's theorem to $F$.

The odd-dimensional case follows from the even-dimensional case and the O'Neill formula, as proved in Grove and Searle \cite{GroveSearle94}). Since the even-dimensional case and O'Neill's formula hold in the weighted case, the proof is complete.
\end{proof}

Finally, it was observed in \cite{Kennard3} that Weinstein's theorem pairs nicely with a result of Davis and Weinberger to provide an obstruction to free group actions on positively curved rational homology spheres of dimension $4k+1$:

\begin{theorem}[Davis--Weinberger factorization]
Let $(M^{4k+1},g,X)$ be closed with positive weighted sectional curvature. If the universal cover of $M$ is a rational homology sphere, then $\pi_1(M) \cong \Z_{2^e} \times \Gamma$ for some odd-order group $\Gamma$.
\end{theorem}

\begin{proof}
Since $\pi_1(M)$ is finite (see Theorem \ref{thm:pi1finite}), we may consider the free action of $\pi_1(M)$ on the universal cover of $M$, which is a compact, simply connected manifold with the same weighted curvature bound as $M$. By Weinstein's theorem in the weighted case, the action of $\pi_1(M)$ is (rationally) homologically trivial. Since $\dim(M) \equiv 1 \bmod{4}$ and the surgery semicharacteristic $\sum_{i \leq 2k} (-1)^i \dim H^i(M;\Q)$ is odd, the factorization of $\pi_1(M)$ follows from Theorem D in \cite{Davis83}.
\end{proof}

\bigskip
\section{Frankel's theorem and Wilking's connectedness lemma}\label{sec:Frankel}
\bigskip

In this section, we prove generalizations of Frankel's theorem and Wilking's connectedness lemma in the weighted case. Specifically, we assume throughout this section that $(M^n,g,X)$ is a Riemannian manifold equipped with a vector field $X$ such that $(M,g,X)$ has positive weighted sectional curvature.

\begin{theorem}[Frankel]\label{thm:Frankel}
Assume $(M^n,g,X)$ is closed with positive weighted sectional curvature. Assume $N_1$ and $N_2$ are closed, totally geodesic submanifolds of $M$ such that $X$ is tangent to $N_i$ for $i\in\{1,2\}$. If $\dim(N_1) + \dim(N_2) \geq n$, then $N_1$ and $N_2$ intersect.
\end{theorem}

Before proving this, we record an easy corollary that we will use in the next section.

\begin{corollary}\label{cor:FrankelGroupAction}
Let $(M^n,g,X)$ be closed with positive weighted sectional curvature. Suppose $G_1$ and $G_2$ are subgroups of the isometry group of $M$, and suppose that $N_1$ and $N_2$ are components of the fixed-point sets of $G_1$ and $G_2$, respectively. If $\dim(N_1) + \dim(N_2) \geq n$, then the submanifolds intersect.
\end{corollary}

To deduce the corollary, one replaces $X$ by $\tilde X$ such that $\tilde X$ is invariant under isometries of $(M,g)$ and $(M,g,\tilde X)$ has positive weighted sectional curvature (see Corollary \ref{cor:PWSCaveraging}). For $p \in N_1$, it follows that $X_p \in (T_p M)^{G_1} = T_p N_1$, hence $X$ is tangent to $N_1$ and likewise for $N_2$. The corollary follows since the $N_i$ are closed and totally geodesic.

\begin{remark}\label{FrankelFailsNoncompact}
Note that both Theorem \ref{thm:Frankel} and the corollary fail if we remove the assumption that $M$ is compact. Indeed, consider the flat metric on Euclidean space, and let $f = \frac 1 2 d^2$, where $d$ is the distance to a fixed point in $M$. Clearly $\sec_f^V(U) = \Hess f(V,V) = 1$ for all orthonormal pairs $(U,V)$, yet any two parallel hyperplanes are disjoint, closed, totally geodesic, and have dimensions adding up to at least $\dim M$.

In fact, $N_1$ and $N_2$ are fixed-point components of reflection subgroups $G_1 \cong \Z_2$ and $G_2 \cong \Z_2$ of the isometry group. However, the subgroup generated by $G_1$ and $G_2$ is infinite, so we cannot replace $X$ by a $G_1$-- and $G_2$--invariant vector field as in Corollary \ref{cor:PWSCaveraging} and proceed as in the proof of the corollary.
\end{remark}

\begin{proof}[Proof of Frankel's theorem]
Let $M$, $N_1$, $N_2$, and $X$ be as in the theorem. We proceed now as in Frankel \cite{Frankel61}. By compactness, there is a minimizing geodesic $\gamma:[a,b] \to M$ connecting $N_1$ to $N_2$. By the first variation formula, $\gamma$ is normal to $N_1$ and $N_2$ at its endpoints. Since $X$ is tangent to $N_1$ and $N_2$,
	\[g\of{\gamma'(b),X_{\gamma(b)}}
	 =g\of{\gamma'(a),X_{\gamma(a)}} = 0.\]
Using Lemma \ref{lem:SecondVariation}, the rest of the proof is as in the unweighted case.
\end{proof}

Wilking proved a vast generalization of Frankel's result (see \cite[Theorem 2.1]{Wilking03}). The generalization to the weighted case is the following:

\begin{theorem}[Wilking's connectedness lemma]\label{thm:Connectedness}
Let $(M^n,g,X)$ be closed with positive weighted sectional curvature.
	\begin{enumerate}
	\item If $X$ is tangent to $N^{n-k}$, a closed, totally geodesic, embedded submanifold of $M$, then the inclusion $N \to M$ is $(n-2k+1)$--connected.
	\item If $X$ and $N^{n-k}$ are as above, and if $G$ acts isometrically on $M$, fixes $N$ pointwise, and has principal orbits of dimension $\delta$, then the inclusion $N \to M$ is $(n-2k+1+\delta)$--connected.
	\item If $X$ is tangent to $N_1^{n-k_1}$ and $N_2^{n-k_2}$, a pair of closed, totally geodesic, embedded submanifolds with $k_1 \leq k_2$, then $N_1 \cap N_2 \to N_2$ is $(n-k_1-k_2)$--connected.
	\end{enumerate}
\end{theorem}

As in the corollary to Frankel's theorem, this result applies to inclusions of fixed-point components of isometric group actions.

\begin{proof}
The proof in each case proceeds as in Wilking \cite[Theorem 2.1]{Wilking03}, where the result is reduced to an index estimate. In the first and third cases, this estimate involves parallel vector fields and hence extends to the weighted case exactly as in the proof of Frankel's theorem above in the weighted case.

In the remaining case, the index estimate is a bit more involved, so we repeat it here, modifying it as necessary to cover the weighted case. The setup in \cite{Wilking03} is as follows: The metric $g_\la$ on $M$ is a Cheeger deformation of $g$, there is a geodesic $c:[a,b] \to M$ that starts and ends perpendicular to $N$, and there is a $(n-2k+1+\delta)$--dimensional vector space $W$ of vector fields $V$ along $c$ such that
	\begin{itemize}
	\item $V$ is tangent to $N$ at the endpoints of $c$,
	\item $V$ is orthgonal to the $G$--orbits at all points along $c$, and
	\item $V'=\nabla_{c'}V$ is tangent to the $G$--orbits at all points.
	\end{itemize}
By the argument in \cite{Wilking03}, it suffices to show that, for all $V \in W$, there exists $\la > 0$ such that the index form with respect to $g_\la$ of $c$ evaluated on $V$ is negative. We show this first under the assumption that $\sec_X > 0$ on $M$.

By Equation \ref{eqn:IndexForm1}, the index form can be written as
	\[\int_a^b \of{|V'|_{g_\la}^2 - (R^{g_\la})_X^{c'}(V,V) - 2g_\la(c',X) g_\la(V,V')} dt + \left.g_\la(c', X)|V|_{g_\la}^2\right|_{t=a}^{t=b}.\]

First, we show that the last term in this expression is zero. Without loss of generality, we may assume that $X$ is $G$--invariant and hence tangent to $N$. Since the $G$--orbits in $N$ are trivial, $X$ is orthogonal to the orbits. Hence the horizontal lift of $X_{c(t)}$ at $t\in\{a,b\}$ is $(0,X_{c(t)})$, and
	\[\left.g_\la(c',X) \right|_{t=a}^{t=b}
	 = \left.g(c',X)\right|_{t=a}^{t=b} = 0.\]

Second, the O'Neill formula in the weighted case implies that $(R^{g_\la})_X^{c'}(V,V) \geq R_X^{c'}(V,V)$. Since this lower bound is independent of $\la > 0$, the proof will be complete once we show both of the following:
	\begin{itemize}
	\item $|V'|_{g_\la}^2 \to 0$ as $\la \to 0$, and
	\item $g_\la(c',X)g_\la(V,V') \to 0$ as $\la \to 0$.
	\end{itemize}
Indeed, since $V'$ is tangent to the $G$--orbits, $|V'|_{g_\la} \to 0$ as $\la \to 0$. This proves the first statement. The second statement follows from the first, together with the estimate
	\[|g_\la(c',X)||g_\la(V,V')| \leq |c'|_{g_\la} |X|_{g_\la}|V|_{g_\la} |V'|_{g_\la} \leq |c'|_g |X|_g|V|_g |V'|_{g_\la}.\]
Here, the second inequality follows since Cheeger deformations (weakly) decrease lengths, i.e., $|\cdot|_{g_\la} \leq |\cdot|_g$ for all $\la > 0$.

This completes the proof if $\sec_X > 0$. Consider now the case where $X = \nabla f$ and $\overline\sec_f > 0$. Here, we consider the vector space
	\[W_f = \{Y = e^f V \st V \in W\},\]
and show that, for all $Y \in W_f$, there exists $\la > 0$ such that the index $I_c(Y,Y)$ of $Y$ along $c$ is negative. Since $\dim(W_f) = \dim(W)$, this would complete the proof in this case. This is easily accomplished by proceeding as in the previous case and using the alternative formula for the index given in Equation \ref{eqn:IndexForm2}.
\end{proof}

\bigskip
\section{Torus actions and positive weighted sectional curvaure}\label{sec:TorusActions}
\bigskip

Throughout this section, we consider closed Riemannian manifolds $(M,g)$ equipped with a vector field $X$ such that $(M,g,X)$ has positive weighted sectional curvature. In addition, we assume a torus $T$ acts isometrically on $M$. Applying Corollary \ref{cor:PWSCaveraging}, if necessary, we assume that $X$ is invariant under the torus action.

Our first result is the following generalization of a result of Grove--Searle \cite{GroveSearle94}:

\begin{theorem}[Maximal symmetry rank]\label{thm:GroveSearle}
Let $(M^n,g,X)$ be closed with positive weighted sectional curvature. If $T^r$ is a torus acting effectively by isometries on $M$, then $r \leq \floor{\frac{n+1}{2}}$. Moreover, if equality holds and $M$ is simply connected, then $M$ is homeomorphic to $\s^n$ or $\C\pp^{n/2}$.
\end{theorem}

The upper bound on $r$ is sharp and agrees with Grove and Searle's result. However, in the unweighted case, Grove and Searle prove an equivariant diffeomorphism classification when the maximal symmetry rank is achieved. We obtain this weaker rigidity statement by a different argument that relies on Wilking's connectedness lemma and a lemma in Fang and Rong \cite{FangRong05}. For a more detailed argument along these lines, we refer to \cite[Section 7.1.3]{PetersenRGv2}

\begin{proof}
By Berger's theorem (Corollary \ref{cor:Berger}) in the weighted case, there exists $x \in M$ fixed by either $T^r$ or a subtorus $T^{r-1}$, according to whether $n$ is even or odd. Since this subtorus embeds into $\SO(n)$ via the isotropy representation, it follows that $r \leq \floor{\frac{n+1}{2}}.$

We proceed to the equality case. First, if $n \in \{2,3\}$, then $M$ is homeomorphic to a sphere since it is simply connected by the resolution of the Poincar\'e conjecture. Suppose therefore that $n \geq 4$. By arguing inductively as in Grove--Searle, it follows that some circle in $T^r$ fixes a codimension-two submanifold $N$. By the connectedness lemma in the weighted case, we conclude that the inclusion $N \embedded M$ is $\dim(N)$--connected. It follows immediately from Poincar\'e duality that $M$ and $N$ are integral cohomology spheres or complex projective spaces (see, for example, \cite[Section 7]{Wilking03}). If $M$ is an integral sphere, then it is a homeomorphism sphere by the resolution of the Poincar\'e conjecture. If $M$ is an integral complex projective space, then the fact that $N$ respresents the generator of $H^2(M;\Z)$ implies that $M$ is homeomorphic to complex projective space, by Lemma 3.6 in Fang--Rong \cite{FangRong05}.
\end{proof}

We remark that there are a number of generalizations of Grove and Searle's result. These include results of Rong and Fang in the cases of ``almost maximal symmetry rank'' or non-negative curvature (see Fang and Rong \cite{Rong02, FangRong05}, Galaz-Garcia and Searle \cite{GalazGarciaSearle11,GalazGarciaSearle14}, and Wiemeler \cite{Wiemeler-pre}).

Returning to the case of positive curvature, there are additional results that assume less symmetry. We focus here on the following homotopy classification due to Wilking \cite[Theorem 2]{Wilking03}:
	\begin{theorem}[Wilking's homotopy classification]
Let $M^n$ be a closed, simply connected, positively curved manifold, and let $T^r$ act effectively by isometries on $M$. If $n \geq 10$ and $r \geq \frac{n}{4} + 1$, then $M$ is either homeomorphic to $\s^n$ or $\HH\pp^{n/4}$ or homotopy equivalent to $\C\pp^{n/2}$.
	\end{theorem}
By Grove and Searle \cite{GroveSearle94} and Fang and Rong \cite{FangRong05}, this result actually holds for all $n \neq 7$. Additionally the conclusion in this theorem has been improved to a classification up to tangential homotopy equivalence (see Dessai and Wilking \cite[Remark 1.4]{DessaiWilking04}). We prove the following analogue of Wilking's classification under a slightly stronger symmetry assumption:

\begin{theorem}\label{thm:WilkingHomotopy}
Let $(M^n,g,X)$ be closed and simply connected with positive weighted sectional curvature. If $M$ admits an effective, isometric torus action of rank $r \geq \frac{n}{4} + \log_2 n$, then $M$ is homeomorphic to $\s^n$ or tangentially homotopy equivalent to $\C\pp^{n/2}$.
\end{theorem}

Note that $\HH\pp^{n/4}$ does not appear in the conclusion. This is consistent with Theorem 3 in Wilking \cite{Wilking03}, which states that the maximal rank of a smooth torus action on an integral $\HH\pp^{m}$ is $m + 1$.

One reason for the larger symmetry assumption is that Wilking's original proof invokes the full strength of Grove and Searle's equivariant diffeomorphism classification. Since we do no prove this here, we cannot use exactly the same proof. In addition, the larger symmetry assumption allows us to side-step some of the more delicate parts of Wilking's proof and thereby allows for a quick argument that captures the essence of his induction machinery, as described in the introduction of \cite{Wilking03}.

\begin{proof}[Proof of Theorem \ref{thm:WilkingHomotopy}]
We first note that it suffices to prove that $M$ has the integral cohomology of $\s^n$ or $\C\pp^{n/2}$. Indeed, a simply connected integral sphere is homeomorphic to the standard sphere by the resolution of the Poincar\'e conjecture. Moreover, it is well known that a simply connected integral complex projective space is homotopy equivalent to the standard one, and the classification up to tangential homotopy follows directly from Dessai and Wilking \cite{DessaiWilking04}.

Second, note that the theorem holds in dimensions $n \leq 13$ by the extension of Grove and Searle's result (Theorem \ref{thm:GroveSearle}). We proceed by induction for dimensions $n \geq 14$. By examining the istropy representation at a fixed-point of $T^r$ (or $T^{r-1}$ in the odd-dimensional case), one sees that an involution $\iota \in T^r$ exists such that some component $N$ of its fixed-point set has codimension $\cod(N) \leq \frac{n+3}{4}$ (see, for example, Lemma \cite[1.8.(1)]{Kennard2}). By replacing $\iota$ by another involution, if necessary, we may assume $\cod(N)$ is minimal. In particular, the induced action of the torus $T^r/\ker(T^r|_N)$ has rank at least $r - 1$.

If $\cod(N) = 2$ and $N$ is fixed by a circle, the claim follows as in the proof of the Grove--Searle result. Otherwise, $T^r/\ker(T^r|_N)$ is a torus that acts effectively and isometrically on $N$ with dimension at least $\frac 1 4 \dim N + \log_2(\dim N)$. Since $N$ is a fixed-point set of an involution in $T^r$, the vector field $X$ is tangent to $N$, and $N$ inherits positive weighted sectional curvature. By the connectedness lemma, $N$ is simply connected. By the induction hypothesis, $N$ is an integral sphere or complex projective space. By the connectedness lemma again, it follows that $M$ too is an integral sphere or projective space. This concludes the proof.
\end{proof}

The theorems of this section should be viewed as a representative, as opposed to exhaustive, list of the kinds of topological results we can now generalize to the weighted setting. Indeed, the tools discussed in this paper have been applied to similar, weaker topological classification problems for positively curved manifolds with torus symmetry. Invariants calculated or estimated include the fundamental group (see Wilking \cite[Theorem 4]{Wilking03}, Frank--Rong--Wang \cite{FrankRongWang13}, Sun--Wang \cite{SunWang09}, and \cite{Kennard3}), the Euler characteristic (see work of the first author and Amann \cite{Kennard1,AK1,AmannKennard3}), and the elliptic genus (see Dessai \cite{Dessai05,Dessai07} and Weisskopf \cite{Weisskopf}). Much of this work now can also be extended to the weighted case using the results in this article.

On the other hand, it is much less clear whether some other prominent classification theorems for manifolds with positive curvature and torus symmetry can be extended to the weighted setting. Principal among these is the situation in low dimensions. In Section \ref{sec:Examples}, we discussed why closed manifolds with positive weighted sectional curvature in dimension two and three are diffeomorphic to spherical space forms. In dimension 4, Hsiang and Kleiner \cite{HsiangKleiner89} proved that a closed, simply connected manifold $M$ in dimension four with positive curvature and an isometric circle action is homeomorphic to $\s^4$ or $\C\pp^2$. 
This result has been generalized in a number of ways. Recently, Grove and Wilking strengthened the conclusion to state that the circle action on $M$ is equivariantly diffeomorphic to a linear action on one of these two spaces (see \cite{GroveWilking-pre} and references therein for a survey of related work). A natural question is whether this result also holds for positive weighted sectional curvature. 

\begin{question}\label{quesdim4}
Let $(M^4,g,X)$ be simply connected and closed with positive weighted sectional curvature. Is every effective, isometric circle action on $M$ equivariantly diffeomorphic to a linear action on $\s^4$ or $\C\pp^2$?
\end{question}

In dimension five, Rong \cite{Rong02} proved that a positively curved $M^5$ with an isometric $2$--torus action is diffeomorphic to $\s^5$. This result has also been improved to an equivariant diffeomorphism classification (see Galaz-Garcia and Searle \cite{GalazGarciaSearle14}), giving the following question. 

\begin{question}\label{quesdim5} 
Let $(M^5,g,X)$ be simply connected and closed with positive weighted sectional curvature. Is every effective, isometric torus action of rank two on $M$ equivariantly diffeomorphic to a linear action on $\s^5$?
\end{question}

\bigskip
\section{Future directions}\label{sec:FutureDirections}
\bigskip

In addition to addressing Questions \ref{quesdim4} and \ref{quesdim5}, another avenue of research is to consider compact manifolds with density that admit positive weighted curvature and an isometric action by an arbitrary Lie group $G$. To make this problem tractible, one can assume that $G$ is large in some sense, e.g., that $G$ or its principal orbits have large dimension. Notable are classification results in this context due to Wallach \cite{Wallach72} and B\'erard-Bergery \cite{Berard-Bergery76} for transitive group actions, Wilking \cite{Wilking06} for more general group actions, Grove and Searle \cite{GroveSearle97} and Spindeler \cite{Spindeler14-thesis} for fixed-point homogeneous group actions, and Grove and Kim \cite{GroveKim04} for fixed-point cohomogeneity one group actions. In the non-negatively curved case, especially in small dimensions, there have been some extensions of these results due to DeVito \cite{DeVito14,DeVito-pre}, Galaz-Garcia and Spindeler \cite{GalazGarciaSpindeler12}, Simas \cite{Simas-pre}, and Gozzi \cite{Gozzi-pre}.

A particularly interesting case is where $G$ is so large that the principal orbits have codimension one. Manifolds that admit a cohomogeneity one metric with positive sectional curvature have been classified by Verdiani \cite{Verdiani04} in the even-dimensional case and by Grove, Wilking, and Ziller \cite{GroveWilkingZiller08} in the odd-dimensional case (see also \cite{VerdianiZiller-pre} and the recent generalization to the case of polar actions by Fang, Grove, and Thorbergsson \cite{FangGroveThorbergsson-pre}). 

The classification is actually incomplete in dimension seven, as there are two infinite families of manifolds that are considered ``candidates'' to admit positive curvature. There are very few examples of manifolds that admit positive curvature, so it was remarkable that one of these candidates was recently shown to admit positive sectional curvature by Dearricott \cite{Dearricott11} and Grove--Verdiani--Ziller \cite{GroveVerdianiZiller11}. It remains to be seen whether the others admit positive curvature.

It would be interesting to examine these results in the case of manifolds with density. Doing this would hopefully lead to new insights into the question posed in the introduction: If $(M,g,X)$ is compact with positive weighted sectional curvature, does $M$ admit a metric with positive sectional cuvature?

The most prominent missing ingredient when trying to generalize results to the weighted setting is a Toponogov-type triangle comparison theorem and the resulting convexity properties of distance functions. These crucial tools would be needed to address Questions \ref{quesdim4} and \ref{quesdim5}, the equivariant diffeomorphism rigidity in Grove and Searle's theorem (Theorem \ref{thm:GroveSearle}), and the results above for general group actions. 

The examples in Section \ref{sec:Examples} show that the classical statement of the Toponogov theorem is false for positive weighted sectional curvature. On the other hand, we can make an analogy here with the situation of Ricci curvature and Bakry--Emery Ricci curvtaure. For positive Ricci curvature, instead of convexity of the distance function, one obtains Laplacian and volume comparisons. These comparisons strictly speaking do not hold for positive Bakry--Emery Ricci curvature, but they have modified weaker versions which are still enough to recover topological obstructions, see \cite{WeiWylie09}. We believe there should be some form of modified convexity for distance functions one obtains from positive weighted sectional curvature which may lead to generalizations of all of the results mentioned above. This will be the topic of future research.


\bibliographystyle{alpha}
\bibliography{myrefs}

\end{document}